\documentclass[a4paper,reqno,10pt]{amsart}

\usepackage{amsmath,amssymb,amsfonts,amsthm,mathrsfs}
\usepackage{verbatim,wasysym,cite}

\usepackage[latin1]{inputenc}
\usepackage{microtype}
\usepackage{color,enumitem,graphicx}

\usepackage[
colorlinks=true,
urlcolor=blue,
citecolor=red,
linkcolor=blue,
linktocpage,
pdfpagelabels,
bookmarksnumbered,
bookmarksopen
]{hyperref}

\usepackage[english]{babel}
\usepackage[symbol]{footmisc}

\renewcommand{\epsilon}{\varepsilon}

\numberwithin{equation}{section}

\newtheorem{theorem}{Theorem}[section]
\newtheorem{lemma}[theorem]{Lemma}
\newtheorem{remark}[theorem]{Remark}
\newtheorem{definition}[theorem]{Definition}
\newtheorem{proposition}[theorem]{Proposition}
\newtheorem{corollary}[theorem]{Corollary}

\newcommand{\C}{\mathbb C}

\newcommand{\R}{\mathbb R}
\newcommand{\N}{\mathbb N}

\def\({\left(}
\def\){\right)}
\def\<{\left\langle}
\def\>{\right\rangle}

\def\Sch{{\mathcal S}}

\def\H{\mathcal H}

\def\d{{\partial}}
\def\eps{\varepsilon}

\def\u{{\tt u}}

\DeclareMathOperator{\Div}{div}
\DeclareMathOperator{\RE}{Re}
\DeclareMathOperator{\IM}{Im}

\begin{document}

\title[Semiclassical cubic--quintic NLS]
{Soliton Dynamics for the Cubic--Quintic Nonlinear
Schr\"odinger Equation with a Potential}

 \author[Alex H. Ardila]{Alex H. Ardila}
\address{Alex H. Ardila
		\newline \indent Department of Mathematics, Universidad del Valle, Colombia} 
\email{ardila@impa.br}

\author[R. Carles]{R\'emi Carles}
\address{R\'emi Carles
		\newline \indent CNRS, Univ. Rennes, IRMAR - UMR 6625,
		F-35000 Rennes, France}
\email{Remi.Carles@math.cnrs.fr}

\begin{abstract}
We study the semiclassical dynamics of solitary waves for the
three-dimensional cubic--quintic nonlinear Schr\"odinger equation with
an external potential. For initial data given by a modulated free
ground state $P_\omega$, we prove that the wave-packet center follows
the associated classical Hamiltonian flow. For $\omega\in\mathcal I$,
we obtain qualitative persistence near the soliton orbit. Under the
additional positive-slope condition $\omega\in\mathcal I_+$, we
establish an $O(\eps)$ approximation in the scaled $H^1$ norm on
every fixed time interval. We also derive $O(\eps^2)$ concentration
estimates for the mass, momentum, and center of mass. 
\end{abstract}



\maketitle

\medskip

\section{Introduction}
\label{intr}

This paper studies the semiclassical dynamics of solitary waves for the
three-dimensional cubic--quintic nonlinear Schr\"odinger equation
under an external potential. For
$0<\eps\leq1$, set
\[
H^\eps:=-\eps^2\Delta+V(x)
\]
and consider the Cauchy problem
\begin{equation}
\label{eqn}
\begin{cases}
i\eps\partial_tu^\eps=H^\eps u^\eps
-|u^\eps|^2u^\eps+|u^\eps|^4u^\eps,
& (t,x)\in\R\times\R^3,\\[1mm]
u^\eps(0,x)
=e^{\frac{i}{\eps}\xi_0\cdot(x-x_0)+i\theta_0}
P_\omega\!\left(\dfrac{x-x_0}{\eps}\right),
& x\in\R^3,
\end{cases}
\end{equation}
where $x_0,\xi_0\in\R^3$ and $\theta_0\in\R$. The soliton-dynamics problem asks whether a wave packet initially localized at
scale $\eps$ preserves its profile and moves, as $\eps\to0$, according to the
classical mechanics generated by the external potential. The pioneering work of
Bronski and Jerrard \cite{BronskiJerrard} combined conservation laws with
variational stability to obtain concentration of mass and momentum and
convergence of a modulated center to the Newtonian trajectory. Keraani
\cite{Keraani} refined this method and proved an $O(\eps)$ approximation in the
scaled $H^1$ norm, with the soliton centered at the exact classical trajectory.
These two results provide the qualitative and quantitative prototypes for the
two theorems stated below.

The same general mechanism has proved robust under substantial changes of the
equation. Squassina \cite{SquassinaMag} incorporated a magnetic field and showed
that the center obeys the Newton--Lorentz law. Montefusco, Pellacci, and
Squassina \cite{MontefuscoPellacciSquassina} treated weakly coupled NLS systems,
where the individual masses and the total momentum concentrate along the
classical flow. Secchi and Squassina \cite{SecchiSquassina} considered fractional
dispersion and exhibited both a modified Newtonian equation and the additional
difficulties caused by the absence of local differential identities. Nonlocal
nonlinearities were addressed by Bonanno, d'Avenia, Ghimenti, and Squassina
\cite{BonannoDaveniaGhimentiSquassina} for generalized Choquard equations,
including situations in which uniqueness and nondegeneracy of ground states are
not known. A different concentration method for strongly nonlinear models was
developed by Benci, Ghimenti, and Micheletti
\cite{BenciGhimentiMicheletti}.
The case of a potential allowed to be singular at one point was
considered by Bonanno \cite{Bonanno}.
For logarithmic NLS, Ardila and Squassina
\cite{ArdilaSquassina} first obtained qualitative dynamics from compactness of
minimizing sequences and then an $O(\eps)$ approximation after establishing the
required quantitative modulation estimate. Finally, although it is not itself a
soliton-dynamics theorem, the work of Zhang, Liu, and Squassina
\cite{ZhangLiuSquassina} shows through the nonlinear Kirchhoff equation how a
spectral coercivity estimate supplies precisely the quantitative modulational
stability needed to close a particle-dynamics argument.

The classical Hamiltonian associated with the principal part of \eqref{eqn} is
\[
h(x,\xi)=|\xi|^2+V(x).
\]
We denote by $(x(t),\xi(t))$ the solution of
\begin{equation}
\label{flow}
\dot x(t)=2\xi(t),
\qquad
\dot\xi(t)=-\nabla V(x(t)),
\qquad
(x(0),\xi(0))=(x_0,\xi_0).
\end{equation}
The corresponding conserved energy is
\[
\mathcal H(t):=\frac12|\xi(t)|^2+\frac12V(x(t))
=\frac12h(x(t),\xi(t)).
\]

Throughout the paper we assume
\begin{equation}
\tag{H1}\label{H1}
V\in C^\infty(\R^3;\R);
\quad
\inf_{x\in\R^3}V(x)>-\infty;
\quad
\forall \alpha\in \N^3,\ |\alpha|\ge 2,\quad \d^\alpha V\in L^\infty(\R^3).
\end{equation}
In particular, $\nabla V$ is globally Lipschitz and has at most linear growth,
while $V$ has at most quadratic growth. Hence \eqref{flow} is globally well
posed. Assumption~\eqref{H1} includes unbounded quadratic
potentials and, in 
particular, harmonic traps (possibly not in all directions,
therefore allowing a partial confinement). Such potentials are relevant in
semiclassical models 
and also occur in the framework of \cite{Keraani}. Since an additive
constant in 
$V$ can be removed from \eqref{eqn} by a time-dependent phase rotation and does
not affect \eqref{flow}, we assume without loss of generality that
$V(x)\geq0$ for $x\in\R^3$.

\begin{remark}\label{rem:smoothness}
  The description of the semiclassical limit $\eps \to 0$ in $u^\eps$,
  that is the proof of Theorems~\ref{t12} and \ref{main} below, uses
  the existence and uniqueness of the solution $u^\eps\in
  C(\R;\Sigma)$ as stated in the conclusion of
  Proposition~\ref{wp} below (including the conservation of mass and
  energy, and 
  the evolution of the center of mass), and, regarding $V$,
  \begin{equation*}
V\in C^3(\R^3;\R),
\qquad
\inf_{x\in\R^3}V(x)>-\infty,
\qquad
\|D^2V\|_{L^\infty}+\|D^3V\|_{L^\infty}<\infty.
\end{equation*}
In  other words, we use the smoothness assumption on $V$ only to
guarantee that the Cauchy problem is solved in a satisfactory way, an
assumption which can certainly be relaxed (at least because the
references \cite{Fujiwara79,Fujiwara,Jao2018} consider finitely many
derivatives of $V$, even though this number does not seem to have been
monitored precisely). 
\end{remark}

For $0<\omega<3/16$, let $P_\omega$ be the positive radial solution of
\begin{equation}
\label{gs}
-\Delta P_\omega+\omega P_\omega-P_\omega^3+P_\omega^5=0,
\qquad P_\omega\in H^1(\R^3).
\end{equation}
Killip, Oh, Pocovnicu, and Vi\c{s}an \cite{KOPV} proved that for each such
$\omega$ there is a unique nonnegative radial solution of \eqref{gs}; it is
strictly positive, radially decreasing, and the map
$\omega\mapsto P_\omega$ is real analytic with values in $H^1(\R^3)$. Their
work also describes the variational and scattering geometry of the free
cubic--quintic flow.

For $u\in H^1(\R^3)$, define
\[
M(u):=\int_{\R^3}|u|^2\,dx,
\qquad
E_0(u):=\frac12\int_{\R^3}|\nabla u|^2\,dx
-\frac14\int_{\R^3}|u|^4\,dx
+\frac16\int_{\R^3}|u|^6\,dx,
\]
and write $m_\omega:=M(P_\omega)$. The fixed-mass variational problem is
\begin{equation}
\label{emin}
E_{\min}(m):=
\inf\bigl\{E_0(u):u\in H^1(\R^3),\ M(u)=m\bigr\}.
\end{equation}
Let $Q_1$ be an optimizer corresponding to $\alpha=1$ in the
Gagliardo--Nirenberg--H\"older inequality of \cite{KOPV}. All such optimizers
have the same mass, so $M(Q_1)$ is well defined. The function $E_{\min}$ is
continuous, concave, nonincreasing, and nonpositive. Moreover,
\[
E_{\min}(m)=0\quad\text{for }0\leq m\leq M(Q_1),
\]
the infimum is not attained for $0<m<M(Q_1)$, and it is attained for
$m\geq M(Q_1)$. For $m>M(Q_1)$ one has $E_{\min}(m)<0$. Every minimizer,
when it exists, belongs up to phase and translation to the soliton family
\cite[Theorem~4.1]{KOPV}.

Uniqueness of $P_\omega$ at a fixed frequency must be distinguished from
uniqueness of the constrained minimizer at a fixed mass: the latter depends on
the global geometry of $\omega\mapsto M(P_\omega)$. To state exactly the
variational hypothesis used in the compactness argument, set
$
\mathcal O_\omega
:=\bigl\{e^{i\theta}P_\omega(\cdot-y):
\theta\in\R,\ y\in\R^3\bigr\}
$
and
\[
\mathcal G(m):=
\bigl\{u\in H^1(\R^3):M(u)=m,\ E_0(u)=E_{\min}(m)\bigr\}.
\]

We define
\begin{equation}
\label{iset}
\mathcal I:=
\left\{
\omega\in\left(0,\tfrac{3}{16}\right):
M(P_\omega)>M(Q_1),\quad
\mathcal G(M(P_\omega))=\mathcal O_\omega
\right\}.
\end{equation}

Thus $P_\omega$ is, for $\omega\in\mathcal I$, the unique fixed-mass energy
minimizer at its mass modulo phase and translation. The strict mass inequality
ensures that the minimum energy is negative, which rules out vanishing in the
compactness argument. As we will see in Remark~\ref{Re11} below, $\mathcal I$ is nonempty. In fact, there exists $\omega_1\in(0,3/16)$ such that
$\left(\omega_1,\tfrac{3}{16}\right)\subset\mathcal I$.

For the quantitative theory we impose the additional condition
\[
\mathcal I_+:=
\left\{
\omega\in\mathcal I:
\tfrac{d}{d\omega}M(P_\omega)>0
\right\}.
\]
This is the variationally admissible part of the positive-slope branch. 

\begin{remark}\label{Re11}
It is worth emphasizing that $\mathcal I_+$ is nonempty. Indeed,  \cite[Theorem~4]{LewinRota} shows that $M(P_\omega)\to+\infty$ and $\tfrac{d}{d\omega}M(P_\omega)>0$,
as $\omega\uparrow3/16$, while \cite[Theorem~7]{LewinRota} establishes
uniqueness, up to phase and translation, of fixed-mass energy
minimizers for all sufficiently large masses. Moreover,
\cite[Corollary~5]{LewinRota} identifies the high-frequency solution
as the positive-slope solution in the three-dimensional cubic--quintic
case. Consequently, there exists $\omega_0\in(0,3/16)$ such that
$\left(\omega_0,\tfrac{3}{16}\right)\subset\mathcal I_+$. 

\end{remark}

We retain the intrinsic formulation in terms of $\mathcal I$ and
$\mathcal I_+$ because 
these are exactly the hypotheses used by our proofs and because they
separate 
the compactness input from the spectral one. Indeed, let
\[
L_+^\omega:=-\Delta+\omega-3P_\omega^2+5P_\omega^4.
\]
Differentiating \eqref{gs} with respect to $\omega$ gives
\[
L_+^\omega\partial_\omega P_\omega=-P_\omega,
\qquad
\langle\partial_\omega P_\omega,P_\omega\rangle_{L^2}
=\tfrac12\tfrac{d}{d\omega}M(P_\omega).
\]
Thus the positive-slope condition expresses the transversality of the soliton
curve to the fixed-mass constraint. It is precisely the condition that yields
the quantitative coercivity of the linearized action, in the spirit of the
modulational estimates initiated by Weinstein \cite{Weinstein}.

The soliton dynamics of Eq. \eqref{eqn}  decorrelates  the role of the
potential and the role of the nonlinearity, as it combines semiclassical
soliton dynamics (involving $V$ through the classical trajectories
\eqref{flow}) with the variational and spectral analysis of
cubic--quintic ground states. The 
conservation-law strategy remains available, but the cubic--quintic
variational 
geometry forces a distinction between compactness-based orbital
stability and 
quantitative spectral coercivity. This distinction is reflected in our
two main 
results: the first applies to the variationally admissible set
$\mathcal I$ and is 
qualitative, whereas the second applies to its positive-slope subset
$\mathcal I_+$ and gives a uniform, quantitative approximation on
every fixed time interval.

For $f\in H^1(\R^3)$, introduce the scaled norm
\[
\|f\|_{H^1_\eps}^2
:=\eps^{-3}\|f\|_{L^2}^2+\eps^{-1}\|\nabla f\|_{L^2}^2.
\]
The scaled mass and energy are
\[
M^\eps(u):=\eps^{-3}\int_{\R^3}|u|^2\,dx
\]
and
\[
E^\eps_{V}(u):=\eps^{-3}\int_{\R^3}
\left[
\frac{\eps^2}{2}|\nabla u|^2+\frac12V(x)|u|^2
-\frac14|u|^4+\frac16|u|^6
\right]dx.
\]
We also use
\[
\Sigma:=\{u\in H^1(\R^3):|x|u\in L^2(\R^3)\}
\]
and the scaled mass and momentum densities
\begin{equation}
\label{dens}
\rho^\eps(t,x):=\eps^{-3}|u^\eps(t,x)|^2,
\qquad
p^\eps(t,x):=2\eps^{-2}\IM
\bigl(\overline{u^\eps(t,x)}\nabla u^\eps(t,x)\bigr).
\end{equation}
The normalization of $p^\eps$ agrees with the classical velocity
$\dot x=2\xi$.

The free cubic--quintic equation is globally well posed in $H^1$ and
$\Sigma$ from \cite{Zhang}. We state here the
corresponding result with an external potential, including all
conservation and balance laws used later.

\begin{proposition}
\label{wp}
Assume \eqref{H1}, let $0<\eps\leq1$, and let $u_0^\eps\in\Sigma$. Then
\[
\begin{cases}
i\eps\partial_tu^\eps+\eps^2\Delta u^\eps=V(x)u^\eps
-|u^\eps|^2u^\eps+|u^\eps|^4u^\eps,\\[1mm]
u^\eps(0)=u_0^\eps,
\end{cases}
\]
has a unique global (strong) solution $u^\eps\in C(\R;\Sigma)\cap
L^6_{\rm loc} (\R;L^{18}(\R^3))$. Mass and energy are conserved,
\[
M^\eps(u^\eps(t))=M^\eps(u_0^\eps),
\qquad
E^\eps_{V}(u^\eps(t))=E^\eps_{V}(u_0^\eps),\quad \forall t\in \R.
\]
Moreover,
\[
\partial_t\rho^\eps+\Div p^\eps=0
\]
in the sense of distributions, and
\[
t\longmapsto\int_{\R^3}p^\eps(t,x)\,dx
\]
is locally absolutely continuous with
\[
\frac{d}{dt}\int_{\R^3}p^\eps(t,x)\,dx
=-2\int_{\R^3}\nabla V(x)\rho^\eps(t,x)\,dx
\]
for almost every $t\in\R$.
\end{proposition}
In the case $V(x)=|x|^2$, this result was proven in
\cite{Jao2016}. It was generalized in \cite{Jao2018} to
potentials satisfying Assumption~\eqref{H1} with in addition
$V(x)\gtrsim |x|^2$. We explain in the appendix how
Proposition~\ref{wp} can be inferred from the approach adopted in
these two references.
\smallbreak

 For the initial
datum in \eqref{eqn}, Proposition~\ref{wp} gives
\begin{equation}
\label{mass}
M^\eps(u^\eps(t))=m_\omega,
\qquad t\in\R,
\end{equation}
and, since $P_\omega$ is real-valued,
\[
\int_{\R^3}p^\eps(0,x)\,dx=2m_\omega\xi_0.
\]

Our first theorem uses only the compactness and uniqueness of fixed-mass
minimizers encoded in $\mathcal I$.

\begin{theorem}[Qualitative semiclassical dynamics]
\label{t12}
Let $\omega\in\mathcal I$. Let $u^\eps$ be the
global solution of \eqref{eqn} given by Proposition~\ref{wp}, and let
$(x(t),\xi(t))$ solve \eqref{flow}. Then there exists $C>0$, independent of
$0<\eps\leq1$, such that
\begin{equation}
\label{e89}
\sup_{t\in\R}\|u^\eps(t)\|_{H^1_\eps}^2\leq C.
\end{equation}
Moreover, for every $\eta>0$ and $T_0>0$ there exists
$\eps_0=\eps_0(\eta,T_0)>0$ such that, for $0<\eps<\eps_0$, there are a time
$T_\eps^*\in(0,T_0]$ and continuous functions
\[
\vartheta^\eps:[0,T_\eps^*)\to\R,
\qquad
z^\eps:[0,T_\eps^*)\to\R^3,
\qquad
w^\eps:[0,T_\eps^*)\to H^1_\eps(\R^3)
\]
for which, with $y^\eps(t):=x(t)+\eps z^\eps(t)$,
\[
u^\eps(t,x)
=e^{\frac{i}{\eps}[\xi(t)\cdot x+\vartheta^\eps(t)]}
P_\omega\!\left(\frac{x-y^\eps(t)}{\eps}\right)
+w^\eps(t,x)
\]
for $0\leq t<T_\eps^*$, and
\begin{equation}
\label{e92}
\sup_{0\leq t<T_\eps^*}\|w^\eps(t)\|_{H^1_\eps}<\eta.
\end{equation}
\end{theorem}

Theorem~\ref{t12} is  qualitative. It asserts closeness to the
soliton orbit but gives no rate in $\eps$; its center is the modulated point
$y^\eps=x+\eps z^\eps$, and the argument only proves that
$T_\eps^*>0$. In particular, it neither shows that $T_\eps^*=T_0$ nor excludes
the possibility that $T_\eps^*\to0$ as $\eps\to0$. This is the same structural
limitation that appears when soliton dynamics is based only on compactness of
minimizing sequences, as in the qualitative logarithmic and fractional results
of \cite{ArdilaSquassina,SecchiSquassina}.

The positive-slope condition removes this obstruction by furnishing a quadratic
modulational estimate. It leads to our main result.

\begin{theorem}[Semiclassical dynamics of a cubic--quintic soliton]
\label{main}
Let $\omega\in\mathcal I_+$. Let $u^\eps$ be the
global solution of \eqref{eqn} given by Proposition~\ref{wp}, and let
$(x(t),\xi(t))$ solve \eqref{flow}. Then, for every $T>0$, there exist
$\eps_T\in(0,1]$ and $C_T>0$ such that, for every $0<\eps<\eps_T$, there are
a continuous phase $\theta^\eps:[0,T]\to\R$ and
$w^\eps\in C([0,T];H^1_\eps(\R^3))$ satisfying
\begin{equation}
\label{deco}
u^\eps(t,x)
=e^{\frac{i}{\eps}\xi(t)\cdot(x-x(t))+i\theta^\eps(t)}
P_\omega\!\left(\frac{x-x(t)}{\eps}\right)
+w^\eps(t,x)
\end{equation}
for every $t\in[0,T]$, with
\begin{equation}
\label{err}
\sup_{0\leq t\leq T}\|w^\eps(t)\|_{H^1_\eps}\leq C_T\eps.
\end{equation}
\end{theorem}

Theorem~\ref{main} strengthens Theorem~\ref{t12} in three precise ways. First,
the approximation holds on any prescribed interval $[0,T]$, with a threshold
$\eps_T$ independent of time inside that interval. Second, the soliton is
centered at the exact classical position $x(t)$ rather than at an uncontrolled
modulated center $x(t)+\eps z^\eps(t)$. Third, the error is $O(\eps)$ in the
scaled $H^1$ norm rather than merely smaller than a fixed tolerance.  Consequently, 
Theorem~\ref{main} is a quantitative
improvement of Theorem~\ref{t12} on $\mathcal I_+$, but it does not replace the
qualitative theorem for frequencies in $\mathcal I\setminus\mathcal I_+$.
Its conclusion is the cubic--quintic analogue of the sharp finite-time
approximation obtained for power nonlinearities in \cite{Keraani}, for magnetic
NLS in \cite{SquassinaMag}, for systems in
\cite{MontefuscoPellacciSquassina}, and for the logarithmic equation in
\cite{ArdilaSquassina}.

The external potential breaks the translation and Galilean symmetries of the
free equation. Thus the parameters in \eqref{deco} are not exact symmetry
parameters of the full flow; they describe the semiclassical modulation of the
localized coherent state. The proof further gives second-order concentration of
the physical densities.

\begin{corollary}[Mass and momentum concentration]
\label{mcon}
Assume the hypotheses of Theorem~\ref{main}. For every $T>0$ there exist
$\eps_T>0$ and $C_T>0$ such that, for $0<\eps<\eps_T$,
\begin{equation}
\label{meas}
\sup_{0\leq t\leq T}
\left[
\left\|\rho^\eps(t)\,dx-m_\omega\delta_{x(t)}\right\|_{(C_b^2)^*}
+
\left\|p^\eps(t)\,dx-2m_\omega\xi(t)\delta_{x(t)}\right\|_{(C_b^2)^*}
\right]
\leq C_T\eps^2.
\end{equation}
\end{corollary}

Define the barycenter of the scaled mass density (or center of mass) by
\[
b^\eps(t):=\frac1{m_\omega}\int_{\R^3}x\rho^\eps(t,x)\,dx.
\]

\begin{corollary}[Barycenter dynamics]
\label{bary}
Assume the hypotheses of Theorem~\ref{main}. For every $T>0$ there exist
$\eps_T>0$ and $C_T>0$ such that, for $0<\eps<\eps_T$,
\begin{equation}
\label{bar2}
\sup_{0\leq t\leq T}|b^\eps(t)-x(t)|\leq C_T\eps^2.
\end{equation}
\end{corollary}

\subsection*{Organization of the paper}

The paper is organized as follows.
The main arguments for the proof of Proposition~\ref{wp} are given in
Appendix~\ref{sec:cauchy}. 
In Section~\ref{prel}, we introduce the notation and develop the
variational and spectral analysis of the free cubic--quintic ground
states. In particular, we prove compactness of fixed-mass minimizing
sequences, establish the coercivity properties of the linearized
action, and derive the qualitative and quantitative modulation
estimates used later. In Section~\ref{SoD}, we obtain the uniform
semiclassical bounds, the energy expansion, and the mass and momentum
identities, and we prove the qualitative dynamics stated in
Theorem~\ref{t12}. Section~\ref{pmt} is devoted to the proof of
Theorem~\ref{main}: we construct the quantitative modulated
decomposition, estimate the localized mass, momentum, and position
defects, and close the argument by a bootstrap and Gronwall's
lemma. Finally, in Section~\ref{pcor}, we prove the mass and momentum
concentration result, and the barycenter estimate, stated in
Corollaries~\ref{mcon} and~\ref{bary}.

\subsection*{Acknowledgments} A.~H.~Ardila was supported by Universidad del Valle through research project CI-71425.

\subsection*{Notation} We use $\|f\|_p:=\|f\|_{L^p(\R^3)}$, $\|f\|_{H^1}^2=\|f\|_2^2+\|\nabla f\|_2^2$, and view $H^1(\R^3;\C)$ 
as a real Hilbert space with $(f,g):=\RE\int f\bar g$. Thus $f\perp g$ means $(f,g)=0$.  For self-adjoint $A$, $(Af,f)$ denotes the quadratic form. Duality $H^{-1}$--$H^1$ is $\langle\cdot,\cdot\rangle$. 
We write $A\lesssim B$ if $A\le C B$, with subscripts indicating dependence (e.g., $\lesssim_\omega$); constants may change line by 
line. Finally, $o_n(1)\to0$ as $n\to\infty$. Notations are adapted so
that when $\eps$ is present as an exponent, its absence corresponds to
setting $\eps=1$, like we did for the mass and the energy above.


\section{Preliminaries}
\label{prel}

In this section, we collect the variational,
spectral, and modulation results needed in the semiclassical
analysis. We first prove compactness of fixed-mass minimizing
sequences. We then study the
linearized action around $P_\omega$ and show that the
positive-slope condition defining $\mathcal I_+$ yields
coercivity transverse to the phase and translation directions.
Finally, we introduce local modulation coordinates and combine
the spectral and variational information to obtain a
quantitative coercivity estimate near the soliton orbit.

\subsection{Variational analysis}
Throughout this subsection, let $\omega\in\mathcal I$. Recall that by the definition of $\mathcal I$, 
$M(P_\omega)>M(Q_1)$, $E_0(P_\omega)=E_{\min}(M(P_\omega))$, and $\mathcal G(M(P_\omega))=\mathcal O_\omega$. 
Thus $P_\omega$ is, modulo phase and translation, the unique minimizer of $E_0$ at mass $M(P_\omega)$.

\begin{lemma}
\label{l23}
Let $\{u_n\}\subset H^1(\R^3)$ satisfy $M(u_n)=M(P_\omega)$ and $E_0(u_n)\longrightarrow E_0(P_\omega)$. 
Then, after passing to a subsequence, there exist
$\{x_n\}\subset\R^3$ and $\{\theta_n\}\subset\R$ such that
\[
e^{-i\theta_n}u_n(\cdot+x_n)
\to
P_\omega
\qquad
\text{strongly in }H^1(\R^3).
\]
\end{lemma}
\begin{proof}
Set $m:=M(P_\omega)$, $e(s):=E_{\min}(s)$. By \cite[Theorem~4.1]{KOPV}, $e$ is continuous and concave, $e=0$ on 
$[0,M(Q_1)]$ and $e<0$ for $s>M(Q_1)$. Since $\omega\in\mathcal I$, $m>M(Q_1)$, hence $e(m)<0$.
Moreover, the above properties of $e$ imply strict subadditivity
\begin{equation}
\label{suba}
e(m)
<
e(m_1)+e(m-m_1),
\qquad
0<m_1<m.
\end{equation}

On the other hand, notice that
\[
E_0(f)+\tfrac{3}{32}M(f)
=
\tfrac12\|\nabla f\|_2^2
+
\tfrac16
\int_{\R^3}
|f|^2
\left(
|f|^2-\tfrac34
\right)^2dx
\]
shows that $\{u_n\}$ is bounded in $H^1$. We apply the $H^1$ profile decomposition from
\cite[Proposition~2.2]{KMV}. After passing to a subsequence, we write
\[
u_n
=
\sum_{j=1}^J
\phi^j(\cdot-x_n^j)
+
w_n^J,
\]
where
\begin{align}
M(u_n)
&=
\sum_{j=1}^J M(\phi^j)
+
M(w_n^J)
+
o_n(1),
\label{pd1}
\\
E_0(u_n)
&=
\sum_{j=1}^J E_0(\phi^j)
+
E_0(w_n^J)
+
o_n(1),
\label{pd2}
\end{align}
and $\lim_{J\to J^*}\limsup_{n\to\infty}\|w_n^J\|_4=0$. In particular, notice that there is at least one nonzero profile. Otherwise
$\|u_n\|_4\to0$, and hence $E_0(u_n)\geq-\frac14\|u_n\|_4^4\to 0$,
contradicting $E_0(u_n)\to e(m)<0$.

Let $\phi^1\neq0$, $m_1:=M(\phi^1)$, $m_r:=m-m_1$. Taking $J=1$ in \eqref{pd1}, $M(w_n^1)\to m_r$. Using \eqref{pd2}, the definition of $e$, and continuity of $e$, we get $e(m)\ge e(m_1)+e(m_r)$. By \eqref{suba}, this is impossible if $m_r>0$. Hence $m_r=0$, $M(\phi^1)=m$, and $M(w_n^1)\to0$.

Since $M(\phi^1)=m$, $E_0(\phi^1)\ge e(m)$. For $n$ large, $M(w_n^1)<M(Q_1)$, so \cite[Theorem~4.1]{KOPV} gives $E_0(w_n^1)\ge0$. The energy decomposition then yields $E_0(\phi^1)=e(m)$ and $E_0(w_n^1)\to0$.

Notice also that for $M(f)\le M(Q_1)$,  we have that (see proof of Theorem~4.1 in \cite{KOPV})

\begin{align}
\label{INED}
	E_0(f)
\ge
\left[
1-
\left(
\tfrac{M(f)}{M(Q_1)}
\right)^{1/2}
\right]
\left[
\tfrac12\|\nabla f\|_2^2
+
\tfrac16\|f\|_6^6
\right].
\end{align}

Applying this to $w_n^1$ shows $\|\nabla w_n^1\|_2\to0$. Therefore, $w_n^1\to0$ strongly in $H^1$. Consequently, $u_n(\cdot+x_n^1)\to\phi^1$ strongly in $H^1$.

Finally, $M(\phi^1)=m$ and $E_0(\phi^1)=E_{\min}(m)$, with $\omega\in\mathcal I$. Hence $\phi^1=e^{i\theta}P_\omega(\cdot-y)$ for some $\theta\in\R$, $y\in\R^3$. This proves the result.
\end{proof}

Throughout the paper we set, for $v\in H^1(\R^3)$,
\[
d_\omega(v):=\inf_{\theta\in\R,\ y\in\R^3}\|e^{-i\theta}v(\cdot+y)-P_\omega\|_{H^1}.
\]

\begin{proposition}
\label{p22}
For every $\eta>0$, there exists $h>0$ such that if $v\in H^1(\R^3)$ satisfies $M(v)=M(P_\omega)$ and $E_0(v)<E_0(P_\omega)+h$, then $d_\omega(v)<\eta$.
\end{proposition}

\begin{proof}
Suppose not. Then for some $\eta_0>0$, there exist $v_n\in H^1$ with $M(v_n)=M(P_\omega)$, $E_0(v_n)<E_0(P_\omega)+1/n$, but $d_\omega(v_n)\ge\eta_0$. Since $\omega\in\mathcal I$, $E_0(P_\omega)=E_{\min}(M(P_\omega))$, so $E_0(v_n)\to E_0(P_\omega)$. Lemma~\ref{l23} then gives $d_\omega(v_n)\to0$, a contradiction.
\end{proof}

\subsection{Spectral analysis}

Define the action  $S_\omega(u):=E_0(u)+\tfrac{\omega}{2}M(u)$. Note
that the stationary equation \eqref{gs} is equivalent to 
$S_\omega'(P_\omega)=0$. Note also that for $w=w_1+iw_2$, $w_1,w_2\in
H^1(\R^3;\R)$ we can write 
\[
(S_\omega''(P_\omega)w,w)
=
q_+^\omega(w_1)
+
q_-^\omega(w_2),
\]
where $q_\pm^\omega(f):=(L_\pm^\omega f,f)$ and $L_+^\omega=-\Delta+\omega-3P_\omega^2+5P_\omega^4$, 
$L_-^\omega=-\Delta+\omega-P_\omega^2+P_\omega^4$. Both operators are self-adjoint on $L^2(\R^3)$ with domain
$H^2(\R^3)$.

The phase and translation symmetries give
$L_-^\omega P_\omega=0$, $L_+^\omega\partial_jP_\omega=0$ for
$j=1,2,3$.
Moreover, by \cite[Propositions~2.4--2.5]{KOPV},
\begin{equation}
\label{kern}
\ker L_+^\omega
=
\operatorname{span}
\{
\partial_1P_\omega,
\partial_2P_\omega,
\partial_3P_\omega
\}.
\end{equation}

Also, by \cite[Theorem~2.2(ii)]{KOPV} the map $\omega\longmapsto P_\omega$
is real analytic. Differentiating \eqref{gs} with respect to
$\omega$ yields
\begin{equation}
\label{dome}
L_+^\omega\partial_\omega P_\omega
=
-P_\omega,
\end{equation}
and
\begin{equation}
\label{mslp}
(\partial_\omega P_\omega,P_\omega)
=
\tfrac12M'(\omega),
\qquad
M'(\omega)
:=
\tfrac{d}{d\omega}M(P_\omega).
\end{equation}

\begin{proposition}
\label{p41}
Let $\omega\in\mathcal I_+$. Then there exists $c_\omega>0$ such that:
\begin{enumerate}[label=\textup{(\roman*)}]
\item if $g\in H^1(\R^3;\R)$ and $g\perp P_\omega$, then $q_-^\omega(g)\ge c_\omega\|g\|_{H^1}^2$;
\item if $f\in H^1(\R^3;\R)$ and $f\perp P_\omega$ and $f\perp \partial_jP_\omega$ for $j=1,2,3$, then $q_+^\omega(f)\ge c_\omega\|f\|_{H^1}^2$;
\item if $w\in H^1(\R^3;\C)$ and $w\perp P_\omega$, $w\perp iP_\omega$, and $w\perp \partial_jP_\omega$ for $j=1,2,3$, then $(S_\omega''(P_\omega)w,w)\ge c_\omega\|w\|_{H^1}^2$.
\end{enumerate}
\end{proposition}

\begin{proof}
First, since $L_-^\omega P_\omega=0$ and $P_\omega>0$, a simple computation gives
\[
q_-^\omega(g)
=
\int_{\R^3}
P_\omega^2
\left|
\nabla\left(\tfrac{g}{P_\omega}\right)
\right|^2dx
\geq0.
\]
In particular, where $P_\omega>0$, $\ker L_-^\omega=\operatorname{span}\{P_\omega\}$. With this in hand, a standard compactness argument yields $C>0$ such that if $f\in H^1(\R^3;\R)$ and $f\perp P_\omega$, then $q_-^\omega(f)\ge C\|f\|_{H^1}^2$. This proves (i).

Next we prove (ii). We first show nonnegativity on the tangent space to the mass constraint. Let $f\in H^1(\R^3;\R)$ satisfy $f\perp P_\omega$ and define
\[
\gamma(s)
:=
\tfrac{\|P_\omega\|_2}
{\|P_\omega+sf\|_2}
(P_\omega+sf).
\]
This curve is well-defined for small $s$, and satisfies $M(\gamma(s))=M(P_\omega)$, $\gamma(0)=P_\omega$, $\gamma'(0)=f$. Since $\omega\in\mathcal I$, $P_\omega$ minimizes $E_0$ at fixed mass, so $S_\omega(\gamma(s))\ge S_\omega(P_\omega)$, which implies $\left.\frac{d^2}{ds^2}S_\omega(\gamma(s))\right|_{s=0}\ge0$.  Thus, the chain rule gives $q_+^\omega(f)\ge0$.

We next exclude a zero direction after removing translations. Suppose $f\perp P_\omega$, $f\perp\partial_jP_\omega$ for $j=1,2,3$, and $q_+^\omega(f)=0$. Let $\mathcal T_\omega:=\{h\in H^1(\R^3;\R):h\perp P_\omega\}$. Since $q_+^\omega\ge0$ on $\mathcal T_\omega$, one can show that
\[
\langle L_+^\omega f,h\rangle=0
\qquad
\text{for every }h\in\mathcal T_\omega.
\]
Thus $L_+^\omega f=cP_\omega$ for some $c\in\R$. Using \eqref{dome}, $L_+^\omega(f+c\partial_\omega P_\omega)=0$. By \eqref{kern}, $f+c\partial_\omega P_\omega\in\operatorname{span}\{\partial_jP_\omega:1\le j\le3\}$. Since $\partial_\omega P_\omega$ is radial, it is orthogonal to each $\partial_jP_\omega$; hence the left-hand side is also orthogonal to the kernel, so $f=-c\partial_\omega P_\omega$. Taking the inner product with $P_\omega$ and using \eqref{mslp}, $0=(f,P_\omega)=-\frac{c}{2}M'(\omega)$. Since $\omega\in\mathcal I_+$, $M'(\omega)>0$, so $c=0$ and thus $f=0$.

Therefore, we have shown that if $f\neq0$ satisfies $f\perp P_\omega$ and $f\perp\partial_jP_\omega$ for $j=1,2,3$, then $q_+^\omega(f)>0$. A standard compactness argument now implies (ii). 

Finally, combining (i) and (ii) immediately yields (iii).
\end{proof}

\begin{lemma}
\label{l45}
There exist $\delta_\omega,C_\omega>0$ such that whenever $\|\phi-P_\omega\|_{H^1}<\delta_\omega$, there exist unique parameters $\Theta(\phi)\in\R$, $Y(\phi)\in\R^3$ in a fixed neighborhood of $(0,0)$ for which $r(\phi):=e^{-i\Theta(\phi)}\phi(\cdot+Y(\phi))-P_\omega$ satisfies $r(\phi)\perp iP_\omega$ and $r(\phi)\perp\partial_jP_\omega$ ($j=1,2,3$). The maps $\Theta,Y$ are $C^1$, and
\[
|\Theta(\phi)|+|Y(\phi)|+\|r(\phi)\|_{H^1}\le C_\omega\|\phi-P_\omega\|_{H^1}.
\]

Consequently, after decreasing $\delta_\omega$ if necessary, if $d_\omega(v)<\delta_\omega$, then there exist $\theta\in\R$, $y\in\R^3$ such that $r:=e^{-i\theta}v(\cdot+y)-P_\omega$ satisfies $r\perp iP_\omega$ and $r\perp\partial_jP_\omega$ ($j=1,2,3$), and
\[
d_\omega(v)\le\|r\|_{H^1}\le C_\omega d_\omega(v).
\]

If $t\mapsto v(t)$ is continuous from an interval $I$ into $H^1$ and $d_\omega(v(t))<\delta_\omega$ for $t\in I$, then the modulation parameters may be chosen continuously on $I$ after fixing one phase branch.
\end{lemma}
\begin{proof}
The proof follows along the same lines as Lemma 6.2 in the Appendix of \cite{LeCoz}. We omit the details.
\end{proof}

\begin{proposition}
\label{stab}
Let $\omega\in\mathcal I_+$. Then there exist $h_\omega,C_\omega>0$ such that if $v\in H^1(\R^3)$ satisfies $M(v)=M(P_\omega)$ and $E_0(v)<E_0(P_\omega)+h_\omega$, then
\begin{equation}
\label{coer}
d_\omega(v)^2
\leq
C_\omega
\left[
E_0(v)-E_0(P_\omega)
\right].
\end{equation}
\end{proposition}
\begin{proof}
We first prove the estimate when $d_\omega(v)$ is sufficiently small. By Lemma~\ref{l45}, there exist $\theta\in\R$ and $y\in\R^3$ such that $r:=e^{-i\theta}v(\cdot+y)-P_\omega$ satisfies $r\perp iP_\omega$, $r\perp\partial_jP_\omega$ ($j=1,2,3$), and $d_\omega(v)\le\|r\|_{H^1}\lesssim_\omega d_\omega(v)$.

Since $M(v)=M(P_\omega)$, we have $2(r,P_\omega)+\|r\|_2^2=0$. Set
\[
a:=\tfrac{(r,P_\omega)}{m_\omega}=-\tfrac{\|r\|_2^2}{2m_\omega},\qquad z:=r-aP_\omega.
\]
Then $z\perp P_\omega,iP_\omega,\partial_jP_\omega$ ($j=1,2,3$), with
$|a|\lesssim_\omega\|r\|_{H^1}^2$, so
$\|z-r\|_{H^1}\lesssim_\omega\|r\|_{H^1}^2$. Note that $z$ satisfies
all the conditions of Proposition~\ref{p41}(iii); moreover, for
$\|r\|_{H^1}$ small, $\|z\|_{H^1}\ge\frac12\|r\|_{H^1}$. 

Then, Proposition~\ref{p41} gives $(S_\omega''(P_\omega)z,z)\ge c_\omega\|z\|_{H^1}^2$. Since $r=z+aP_\omega$ and $a=O(\|r\|_{H^1}^2)$, we get $(S_\omega''(P_\omega)r,r)\ge c_\omega\|r\|_{H^1}^2$ after shrinking $\|r\|_{H^1}$ if needed.

Using $S_\omega'(P_\omega)=0$, Taylor's formula yields
\[
S_\omega(P_\omega+r)-S_\omega(P_\omega)
=
\frac12(S_\omega''(P_\omega)r,r)+o(\|r\|_{H^1}^2).
\]
Thus $S_\omega(P_\omega+r)-S_\omega(P_\omega)\gtrsim_\omega\|r\|_{H^1}^2$. Since $M(P_\omega+r)=M(P_\omega)$, the mass terms cancel, so $E_0(v)-E_0(P_\omega)\gtrsim_\omega\|r\|_{H^1}^2\ge d_\omega(v)^2$.

Finally, choose $\eta_\omega>0$ small enough that all local estimates hold whenever $d_\omega(v)<\eta_\omega$. Proposition~\ref{p22} with $\eta=\eta_\omega$ yields $h_\omega>0$ such that $M(v)=M(P_\omega)$ and $E_0(v)<E_0(P_\omega)+h_\omega$ imply $d_\omega(v)<\eta_\omega$. This proves the proposition.
\end{proof}

\section{Soliton Dynamics}
\label{SoD}

The main purpose of this section is to prove the qualitative soliton decomposition of Theorem~\ref{t12}. For this, we establish several lemmas.

\begin{lemma}
\label{l31}
Let $u^\eps$ be the global solution
to \eqref{eqn}. Then
\[
\sup_{t\in\R}
\|\nabla u^\eps(t)\|_2^2
\lesssim
\eps
\]
uniformly for $0<\eps\leq1$. Moreover, $\sup_{t\in\R}\|p^\eps(t)\|_1\lesssim1$.
\end{lemma}
\begin{proof}
From \eqref{INED}, conservation of mass and energy give (recall that $V\geq 0$)
\[
\tfrac{1}{2\eps}\|\nabla u^\eps(t)\|_2^2
\leq
E^\eps_{V}(u^\eps(0))
+
\tfrac{3}{32}m_\omega.
\]
Notice also that
\[
E^\eps_{V}(u^\eps(0))
=
E_0(P_\omega)
+
\tfrac12m_\omega|\xi_0|^2
+
\tfrac12
\int_{\R^3}
V(x_0+\eps y)P_\omega^2(y)\,dy.
\]
Since $D^2V\in L^\infty$, $V$ has at most quadratic growth, so $|V(x_0+\eps y)|\lesssim1+|y|^2$ for $0<\eps\le1$, with constant depending on $x_0,V$. The exponential decay of $P_\omega$ yields $E^\eps_{V}(u^\eps(0))\lesssim1$, hence $\|\nabla u^\eps(t)\|_2^2\lesssim\eps$.

Finally, by definition of $p^\eps$, conservation of mass, and the preceding estimate, $\|p^\eps(t)\|_1\le2\eps^{-2}\|u^\eps(t)\|_2\|\nabla u^\eps(t)\|_2\lesssim1$.
\end{proof}

\begin{lemma}
\label{l32}
Let $f\in C^2(\R^3)$ with $\|D^2f\|_{L^\infty}<\infty$. Then, as $\eps\to0$,
\[
\int_{\R^3}f(y+\eps x)P_\omega^2(x)\,dx=m_\omega f(y)+O(\eps^2),
\]
uniformly in $y\in\R^3$.
\end{lemma}

\begin{proof}
Taylor's formula gives $f(y+\eps x)=f(y)+\eps\nabla f(y)\cdot x+R_\eps(x,y)$, with $|R_\eps|\le\frac12\eps^2\|D^2f\|_\infty|x|^2$. Since $P_\omega$ is radial, $\int xP_\omega^2=0$. Thus
\[
\left|\int f(y+\eps x)P_\omega^2-m_\omega f(y)\right|
\lesssim\eps^2\int|x|^2P_\omega^2,
\]
and the last integral is finite by exponential decay.
\end{proof}

\begin{lemma}
\label{l33}
Let $u^\eps$ be the global solution to \eqref{eqn}. Then
\[
E^\eps_{V}(u^\eps(t))
=
E_0(P_\omega)
+
m_\omega\mathcal H(t)
+
O(\eps^2)
\]
uniformly for $t\in\R$.
\end{lemma}

\begin{proof}
Conservation of semiclassical energy gives
$E^\eps_{V}(u^\eps(t))=E^\eps_{V}(u^\eps(0))$. Recalling 
\[
E^\eps_{V}(u^\eps(0))
=
E_0(P_\omega)
+
\tfrac12m_\omega|\xi_0|^2
+
\tfrac12
\int_{\R^3}
V(x_0+\eps y)P_\omega^2(y)\,dy,
\]
Lemma~\ref{l32} with $f=V$ yields $\int V(x_0+\eps y)P_\omega^2=m_\omega V(x_0)+O(\eps^2)$. Hence $E^\eps_{V}(u^\eps(t))=E_0(P_\omega)+m_\omega\mathcal H(0)+O(\eps^2)$. Since $\mathcal H$ is conserved along \eqref{flow}, the result follows.
\end{proof}

Fix $T_0>0$ and set
$R_0:=\sup_{0\leq t\leq T_0}|x(t)|$, and 
$\rho:=R_0+1$.  Choose $\chi\in C_c^\infty(\R^3)$ such that
\begin{equation}
\label{e73}
0\leq\chi\leq1,
\qquad
\chi=1
\ \text{on }B(0,\rho),
\qquad
\chi=0
\ \text{on }\R^3\setminus B(0,2\rho).
\end{equation}

We introduce the momentum and potential-energy defects
\begin{equation}
\label{e74}
\sigma^\eps(t)
:=
\int_{\R^3}
p^\eps(t,x)\,dx
-
2m_\omega\xi(t)
\end{equation}
and
\begin{equation}
\label{e75}
\lambda^\eps(t)
:=
m_\omega V(x(t))
-
\int_{\R^3}
\chi(x)V(x)\rho^\eps(t,x)\,dx.
\end{equation}

\begin{lemma}
\label{l34}
Let $u^\eps$ be the global solution to \eqref{eqn}. Then $\sigma^\eps\in C([0,T_0];\R^3)$, $\lambda^\eps\in C([0,T_0];\R)$, with $\sigma^\eps(0)=0$ and $|\lambda^\eps(0)|\lesssim\eps^2$ for all sufficiently small $\eps>0$.
\end{lemma}

\begin{proof}
Since $u^\eps\in C(\R;H^1)$, the total momentum $t\mapsto\int p^\eps(t,x)\,dx$ is continuous; hence $\sigma^\eps$ is continuous. 
Similarly, $\chi V$ is bounded and compactly supported, so $t\mapsto\int\chi V\,\rho^\eps(t)$ is continuous, and thus $\lambda^\eps\in C([0,T_0];\R)$.
Moreover, since $\int p^\eps(0,x)\,dx=2m_\omega\xi_0$, we get $\sigma^\eps(0)=0$.

On the other hand, changing variables gives
\[
\lambda^\eps(0)
=
m_\omega V(x_0)
-
\int V(x_0+\eps y)P_\omega^2(y)\,dy
+
\int [1-\chi(x_0+\eps y)]V(x_0+\eps y)P_\omega^2(y)\,dy.
\]
By Lemma~\ref{l32}, the first difference is $O(\eps^2)$. If $1-\chi(x_0+\eps y)\neq0$, then $x_0+\eps y\notin B(0,\rho)$, so $|x_0+\eps y|\ge\rho$. Since $|x_0|\le R_0=\rho-1$, we get $\eps|y|\ge |x_0+\eps y|-|x_0|\ge\rho-R_0=1$, hence $|y|\ge\eps^{-1}$.

Moreover, boundedness of $D^2V$ gives $|V(x_0+\eps y)|\lesssim1+\eps|y|+\eps^2|y|^2$. In the region $|y|\ge\eps^{-1}$, we have $1+\eps|y|+\eps^2|y|^2\le3\eps^2|y|^2$. Thus
\[
\left|\int [1-\chi(x_0+\eps y)]V(x_0+\eps y)P_\omega^2(y)\,dy\right|
\le C\int_{|y|\ge\eps^{-1}}(1+\eps|y|+\eps^2|y|^2)P_\omega^2(y)\,dy
\lesssim_\omega\eps^2.
\]
Therefore, $|\lambda^\eps(0)|\lesssim\eps^2$. This completes the proof.
\end{proof}

Define
\begin{equation}
\label{psid}
\psi^\eps(t,x)
:=
e^{-\frac{i}{\eps}
\xi(t)\cdot[\eps x+x(t)]}
u^\eps(t,\eps x+x(t)).
\end{equation}
A change of variables and conservation of mass give
\begin{equation}
\label{psim}
M(\psi^\eps(t))
=
M^\eps(u^\eps(t))
=
m_\omega.
\end{equation}

\begin{lemma}
\label{l35}
Let $\omega\in\mathcal I$, and let $u^\eps$ be the global solution to \eqref{eqn}. Then, for every $t\in[0,T_0]$,
\begin{align}
E_0(\psi^\eps(t))-E_0(P_\omega)
&=
\frac12\lambda^\eps(t)
-\frac12\xi(t)\cdot\sigma^\eps(t)
\nonumber\\
&\quad
-\frac12\int_{\R^3}(1-\chi(x))V(x)\rho^\eps(t,x)\,dx
+O(\eps^2),
\label{l35e}
\end{align}
and
\[
0\le E_0(\psi^\eps(t))-E_0(P_\omega)
\le
\frac12|\xi(t)|\,|\sigma^\eps(t)|
+\frac12|\lambda^\eps(t)|
+C\eps^2,
\]
uniformly in $t$.
\end{lemma}

\begin{proof}
From \eqref{psid}, $\nabla\psi^\eps(x,t)=e^{-\frac{i}{\eps}\xi(t)\cdot[\eps x+x(t)]}[\eps\nabla u^\eps(y,t)-i\xi(t)u^\eps(y,t)]$, $y=\eps x+x(t)$. Changing variables in the  energy gives
\[
E_0(\psi^\eps(t))
=
E^\eps_{V}(u^\eps(t))
-\frac12\int V(x)\rho^\eps(t,x)\,dx
+\frac12m_\omega|\xi(t)|^2
-\frac12\xi(t)\cdot\int p^\eps(t,x)\,dx.
\]
Using Lemma~\ref{l33}, $\mathcal H(t)=\frac12|\xi(t)|^2+\frac12V(x(t))$, and $\int p^\eps=2m_\omega\xi+\sigma^\eps$, we obtain
\[
E_0(\psi^\eps(t))-E_0(P_\omega)
=
\frac12m_\omega V(x(t))
-\frac12\int V(x)\rho^\eps(t,x)\,dx
-\frac12\xi(t)\cdot\sigma^\eps(t)
+O(\eps^2).
\]
Splitting $V=\chi V+(1-\chi)V$ and using \eqref{e75} yields \eqref{l35e}.

Since $V\ge0$ and $0\le\chi\le1$, \eqref{l35e} gives
\[
E_0(\psi^\eps(t))-E_0(P_\omega)
\le
\frac12|\lambda^\eps(t)|
+\frac12|\xi(t)|\,|\sigma^\eps(t)|
+O(\eps^2).
\]
On the other hand, by \eqref{psim} and $\omega\in\mathcal I$, $M(\psi^\eps(t))=M(P_\omega)$ and $E_0(P_\omega)=E_{\min}(M(P_\omega))$, so $E_0(\psi^\eps(t))-E_0(P_\omega)\ge0$. This completes the proof.
\end{proof}

\begin{lemma}
\label{l36}
Let $L:=\|D^2V\|_{L^\infty}$. Then 
\begin{equation}
\label{vgr}
|\nabla V(x)|^2
\leq
2LV(x),
\qquad
x\in\R^3.
\end{equation}
If $L=0$, then $\nabla V\equiv0$.
\end{lemma}
\begin{proof}
The proof is standard. For completeness, we give the details. Suppose first that $L>0$. Since $\nabla V$ is $L$-Lipschitz,
\[
V(y)\le V(x)+\nabla V(x)\cdot(y-x)+\tfrac L2|y-x|^2.
\]
Taking $y=x-\frac1L\nabla V(x)$ gives
$
0\le V(y)\le V(x)-\tfrac{1}{2L}|\nabla V(x)|^2,
$
which proves \eqref{vgr}. If $L=0$, then $V$ is affine; since $V$ is bounded below on $\R^3$, its linear part must vanish.
\end{proof}

We now give the proof of Theorem~\ref{t12}.

\begin{proof}[Proof of Theorem~\ref{t12}]
We follow closely the argument in \cite[Theorem 4.2]{SquassinaMag}. Estimate \eqref{e89} follows from Lemma~\ref{l31}. We prove the soliton decomposition.

Fix $\eta>0$ and $T_0>0$. Since $V\ge0$, we see that $|\xi(t)|^2\le2\mathcal H(0)$. Thus, we define $K_\xi:=\sup_{t\in\R}|\xi(t)|<\infty$. Let $\delta_\omega>0$ be the modulation radius in Lemma~\ref{l45}. Choose $\eta_0>0$ such that $\eta_0<\delta_\omega$ and $C_\xi C_\omega\eta_0<\eta$, where $C_\omega$ is from Lemma~\ref{l45} and $C_\xi=C(K_\xi)$ satisfies $\|r\|_2^2+\|\nabla r+i\xi r\|_2^2\le C_\xi^2\|r\|_{H^1}^2$ whenever $|\xi|\le K_\xi$.

Let $h>0$ be from Proposition~\ref{p22} for $\eta_0$, and define
$\beta^\eps(t):=\tfrac12|\xi(t)|\,|\sigma^\eps(t)|+
\tfrac12|\lambda^\eps(t)|$. Note that 
$\beta^\eps$ is continuous, and by Lemma~\ref{l34}, $\beta^\eps(0)\leq
C\eps^2$ for some $C>0$. Set 
\[
T_\eps^*:=\sup\{t\in[0,T_0]:\beta^\eps(s)\le \tfrac h2\text{ for all }s\in[0,t]\}.
\]
For $\eps$ small, it is not hard to show that $T_\eps^*>0$. If $t<T_\eps^*$, Lemma~\ref{l35} gives
\[
0\le E_0(\psi^\eps(t))-E_0(P_\omega)\le \tfrac h2+C\eps^2.
\]
Taking $\eps$ smaller if needed, we infer that $E_0(\psi^\eps(t))-E_0(P_\omega)<h$. Since $M(\psi^\eps(t))=M(P_\omega)$, Proposition~\ref{p22} yields $d_\omega(\psi^\eps(t))<\eta_0$.

On the other hand, by Lemma~\ref{l45} we have continuous parameters
$\theta_*^\eps(t)\in\R$, $z^\eps(t)\in\R^3$ such that
$r^\eps(t,x):=e^{-i\theta_*^\eps(t)}\psi^\eps(t,x+z^\eps(t))-P_\omega(x)$
satisfies $r^\eps(t)\perp iP_\omega$,
$r^\eps(t)\perp\partial_jP_\omega$ ($j=1,2,3$), and 
\begin{align}
	\|r^\eps(t)\|_{H^1}\le C_\omega d_\omega(\psi^\eps(t))<C_\omega\eta_0. \label{INE3}
\end{align}
Set $y^\eps(t):=x(t)+\eps z^\eps(t)$, $\vartheta^\eps(t):=\eps\theta_*^\eps(t)$. Recalling \eqref{psid}, we get
\[
u^\eps(t,x)=e^{\tfrac{i}{\eps}[\xi(t)\cdot x+\vartheta^\eps(t)]}P_\omega\!\left(\tfrac{x-y^\eps(t)}{\eps}\right)+w^\eps(t,x),
\]
where $w^\eps(t,X):=e^{\tfrac{i}{\eps}[\xi(t)\cdot X+\vartheta^\eps(t)]}r^\eps\!\left(t,\tfrac{X-y^\eps(t)}{\eps}\right)$. A direct computation gives
\[
\|w^\eps(t)\|_{H^1_\eps}^2
=
\|r^\eps(t)\|_2^2+\|\nabla r^\eps(t)+i\xi(t)r^\eps(t)\|_2^2
\le C_\xi^2\|r^\eps(t)\|_{H^1}^2.
\]
Thus by \eqref{INE3}, $\sup_{0\le t<T_\eps^*}\|w^\eps(t)\|_{H^1_\eps}<\eta$, which proves the theorem.
\end{proof}


\section{Proof of  Theorem~\ref{main}}
\label{pmt}

The main objective of this section is to prove Theorem~\ref{main}.
Throughout this section, $T_0>0$ is fixed, and $\chi$ denotes
the cutoff introduced in \eqref{e73}.

\begin{proposition}
\label{p51}
Let $\omega\in\mathcal I_+$ and let $I\subset[0,T_0]$ be an interval containing $0$ such that
\[
E_0(\psi^\eps(t))<E_0(P_\omega)+h_\omega,\qquad t\in I,
\]
where $h_\omega$ is from Proposition~\ref{stab}. Then there exist continuous $\theta_*^\eps:I\to\R$, $z^\eps:I\to\R^3$, with $z^\eps(0)=0$, such that for $y^\eps(t):=x(t)+\eps z^\eps(t)$,
\[
u^\eps(t,x)=e^{\frac{i}{\eps}[\xi(t)\cdot x+\eps\theta_*^\eps(t)]}P_\omega\!\left(\frac{x-y^\eps(t)}{\eps}\right)+w_1^\eps(t,x),
\]
and
\[
\|w_1^\eps(t)\|_{H^1_\eps}^2\lesssim_{T_0}|\sigma^\eps(t)|+|\lambda^\eps(t)|+\eps^2,\qquad t\in I.
\]
\end{proposition}

\begin{proof}
By Proposition~\ref{stab} and Lemma~\ref{l35},
\[
d_\omega(\psi^\eps(t))^2
\lesssim_\omega E_0(\psi^\eps(t))-E_0(P_\omega)
\lesssim_{T_0}|\sigma^\eps(t)|+|\lambda^\eps(t)|+\eps^2,
\]
using the uniform bound on $\xi(t)$. Lemma~\ref{l45} gives continuous $\theta_*^\eps(t)$, $z^\eps(t)$ such that
\[
r^\eps(t,x):=e^{-i\theta_*^\eps(t)}\psi^\eps(t,x+z^\eps(t))-P_\omega(x)
\]
satisfies $r^\eps(t)\perp iP_\omega$, $r^\eps(t)\perp\partial_jP_\omega$ ($j=1,2,3$), and
\[
\|r^\eps(t)\|_{H^1}^2\lesssim_{T_0}|\sigma^\eps(t)|+|\lambda^\eps(t)|+\eps^2.
\]
At $t=0$, $\psi^\eps(0)=e^{i(\theta_0-\xi_0\cdot x_0/\eps)}P_\omega$, so we choose $z^\eps(0)=0$. Notice also that
\[
u^\eps(t,x)=e^{\frac{i}{\eps}[\xi(t)\cdot x+\eps\theta_*^\eps(t)]}
\left[P_\omega\!\left(\frac{x-y^\eps(t)}{\eps}\right)
+r^\eps\!\left(t,\frac{x-y^\eps(t)}{\eps}\right)\right].
\]
Setting $w_1^\eps(t,x):=e^{\frac{i}{\eps}[\xi(t)\cdot x+\eps\theta_*^\eps(t)]}r^\eps(t,\frac{x-y^\eps(t)}{\eps})$, a direct computation gives
\[
\|w_1^\eps(t)\|_{H^1_\eps}^2
=
\|r^\eps(t)\|_2^2+\|\nabla r^\eps(t)+i\xi(t)r^\eps(t)\|_2^2
\lesssim\|r^\eps(t)\|_{H^1}^2
\lesssim|\sigma^\eps(t)|+|\lambda^\eps(t)|+\eps^2,
\]
which proves the proposition.
\end{proof}

\begin{lemma}
\label{l52}
Let $I$ be as in Proposition~\ref{p51} and set $D^\eps(t):=|\sigma^\eps(t)|+|\lambda^\eps(t)|+\eps^2$. Assume $D^\eps(t)\le1$ for $t\in I$. Then
\[
\left\|\rho^\eps(t)\,dx-m_\omega\delta_{y^\eps(t)}\right\|_{(C_b^2)^*}
+
\left\|p^\eps(t)\,dx-2m_\omega\xi(t)\delta_{y^\eps(t)}\right\|_{(C_b^2)^*}
\lesssim_{T_0}D^\eps(t)
\]
for every $t\in I$.
\end{lemma}
\begin{proof}
This is the nonmagnetic counterpart of \cite[Lemma~6.1]{SquassinaMag}; the proof is essentially the same, so we omit the details.
\end{proof}

For the next lemmas, we introduce the following quantities:
\[
\gamma^\eps(t)
:=
m_\omega x(t)
-
\int_{\R^3}
x\chi(x)\rho^\eps(t,x)\,dx
\]
and
\[
\Xi^\eps(t)
:=
|\sigma^\eps(t)|
+
|\lambda^\eps(t)|
+
|\gamma^\eps(t)|.
\]

\begin{lemma}
\label{l53}
The function $\gamma^\eps:[0,T_0]\to\R^3$ is continuous and satisfies $|\gamma^\eps(0)|\lesssim_{T_0}\eps^2$ for all sufficiently small $\eps$.
\end{lemma}

\begin{proof}
Set $a(x):=x\chi(x)$. Since $t\mapsto\rho^\eps(t)$ is continuous as an $L^1$-valued map and $a(x)$ is bounded, $\gamma^\eps\in C([0,T_0];\R^3)$. 

On the other hand, at $t=0$, changing variables $x=x_0+\eps y$ gives
\[
\int a(x)\rho^\eps(0,x)\,dx=\int a(x_0+\eps y)P_\omega^2(y)\,dy.
\]
Applying Lemma~\ref{l32} componentwise yields $\int a(x_0+\eps y)P_\omega^2=m_\omega a(x_0)+O(\eps^2)$. Since $|x_0|\le R_0<\rho$, we have $\chi(x_0)=1$, so $a(x_0)=x_0$. Thus
\[
\gamma^\eps(0)=m_\omega x_0-m_\omega a(x_0)+O_{T_0}(\eps^2)=O(\eps^2).
\]
\end{proof}

\begin{lemma}
\label{l54}
There exist $h_0=h_0(T_0)>0$ and $\eps_0=\eps_0(T_0)>0$ such that the following holds. Let $I=[0,\tau)\subset[0,T_0]$ be an interval where Proposition~\ref{p51} and Lemma~\ref{l52} apply, and suppose $\sup_{t\in I}\Xi^\eps(t)\le h_0$. Then, for $0<\eps<\eps_0$, $\chi(y^\eps(t))=1$ and $|x(t)-y^\eps(t)|\lesssim_{T_0}\Xi^\eps(t)+\eps^2$ for every $t\in I$.
\end{lemma}

\begin{proof}
By our modulation branch choice, $y^\eps(0)=x_0$. Recall $R_0=\sup_{0\le t\le T_0}|x(t)|$, $\rho=R_0+1$. Suppose $y^\eps$ leaves $B(0,\rho)$, and let $t_*\in I$ be the first exit time. Then $|y^\eps(t_*)|=\rho$, $|y^\eps(t)|\le\rho$ for $0\le t\le t_*$, so $\chi(y^\eps(t))=1$ on $[0,t_*]$.

Set $a(x):=x\chi(x)$. Since $a(y^\eps(t))=y^\eps(t)$ on this interval,
\[
m_\omega[x(t)-y^\eps(t)]
=
\gamma^\eps(t)+\int a(x)\rho^\eps(t,x)\,dx-m_\omega a(y^\eps(t)).
\]
Since $a\in C_b^2(\R^3;\R^3)$, Lemma~\ref{l52} gives
\[
|x(t)-y^\eps(t)|\lesssim_{T_0}\Xi^\eps(t)+\eps^2,\qquad 0\le t\le t_*.
\]
Choose $h_0$ and $\eps_0$ so $C_{T_0}(h_0+\eps_0^2)<\tfrac12$. At $t=t_*$, this gives $|x(t_*)-y^\eps(t_*)|<\tfrac12$. But $|y^\eps(t_*)|=R_0+1$ and $|x(t_*)|\le R_0$, so $|y^\eps(t_*)-x(t_*)|\ge1$, a contradiction. Hence no exit occurs, and the estimate holds for $t  \in I$.
\end{proof}

\begin{lemma}
\label{l55}
Under the hypotheses of Lemma~\ref{l54},
\[
\left\|\rho^\eps(t)\,dx-m_\omega\delta_{x(t)}\right\|_{(C_b^2)^*}
+
\left\|p^\eps(t)\,dx-2m_\omega\xi(t)\delta_{x(t)}\right\|_{(C_b^2)^*}
\lesssim_{T_0}\Xi^\eps(t)+\eps^2
\]
for every $t\in I$.
\end{lemma}

\begin{proof}
For $a,b\in\R^3$, $\|\delta_a-\delta_b\|_{(C_b^2)^*}\le|a-b|$. Thus, by Lemmas~\ref{l52} and~\ref{l54},
\[
\left\|\rho^\eps(t)\,dx-m_\omega\delta_{x(t)}\right\|_{(C_b^2)^*}
\le
\left\|\rho^\eps(t)\,dx-m_\omega\delta_{y^\eps(t)}\right\|_{(C_b^2)^*}
+m_\omega|y^\eps(t)-x(t)|
\lesssim_{T_0}\Xi^\eps(t)+\eps^2.
\]
Similarly,
\[
\left\|p^\eps(t)\,dx-2m_\omega\xi(t)\delta_{x(t)}\right\|_{(C_b^2)^*}
\le
\left\|p^\eps(t)\,dx-2m_\omega\xi(t)\delta_{y^\eps(t)}\right\|_{(C_b^2)^*}
+2m_\omega|\xi(t)|\,|y^\eps(t)-x(t)|.
\]
Since $\xi$ is uniformly bounded, the conclusion follows from Lemmas~\ref{l52} and~\ref{l54}.
\end{proof}

We now give the proof of Theorem~\ref{main}.

\begin{proof}[Proof of Theorem~\ref{main}]
Fix $T>0$ and take $T_0=T$. Set $\mathcal D^\eps(t):=\Xi^\eps(t)+\eps^2$. By Lemma~\ref{l35} and the uniform boundedness of $\xi$,
\[
0\le E_0(\psi^\eps(t))-E_0(P_\omega)\le A_T\mathcal D^\eps(t)
\]
for some $A_T>0$. Let $h_0,\eps_0$ be from Lemma~\ref{l54}, and let $h_\omega$ be from Proposition~\ref{stab}. Shrinking $h_0,\eps_0$ if needed, assume $h_0+\eps_0^2<1$ and $A_T(h_0+\eps_0^2)<h_\omega$. By Lemmas~\ref{l34} and~\ref{l53}, $\Xi^\eps(0)\lesssim_T\eps^2$. Choose $\eps_T\in(0,\eps_0)$ such that $\Xi^\eps(0)<\tfrac12h_0$ for $0<\eps<\eps_T$.

Define
\[
\tau_\eps^*:=\sup\{t\in[0,T]:\Xi^\eps(s)\le h_0\text{ for all }s\in[0,t]\}.
\]
By continuity of $\sigma^\eps,\lambda^\eps,\gamma^\eps$, we have that $\tau_\eps^*>0$. For $t<\tau_\eps^*$,  we see that $\Xi^\eps(t)\le h_0$, hence
\begin{align}
	E_0(\psi^\eps(t))<E_0(P_\omega)+h_\omega. \label{ECC}
\end{align}

Thus Proposition~\ref{p51} applies on $[0,\tau_\eps^*)$. Moreover, $|\sigma^\eps|+|\lambda^\eps|+\eps^2\le\mathcal D^\eps(t)<1$, so Lemmas~\ref{l52}--\ref{l55} apply. In particular,
\begin{align}
	\left\|\rho^\eps(t)\,dx-m_\omega\delta_{x(t)}\right\|_{(C_b^2)^*}
+
\left\|p^\eps(t)\,dx-2m_\omega\xi(t)\delta_{x(t)}\right\|_{(C_b^2)^*}
\lesssim_T\mathcal D^\eps(t). \label{conc}
\end{align}

We now estimate the three components of $\Xi^\eps$.
Indeed, we claim that
$  
|\dot{\sigma}^\eps(t)|
+|\dot{\lambda}^\eps(t)|
+|\dot{\gamma}^\eps(t)|
\lesssim_T \mathcal D^\eps(t).
$ 
By the momentum identity and $\dot\xi=-\nabla V(x(t))$,
\[
\dot\sigma^\eps(t)=-2\int\nabla V\,\rho^\eps\,dx+2m_\omega\nabla V(x(t)).
\]
Since $\chi(x(t))=1$,
\[
\tfrac12\dot\sigma^\eps(t)
=
-\left[\int\chi\nabla V\,\rho^\eps\,dx-m_\omega\nabla V(x(t))\right]
-\int(1-\chi)\nabla V\,\rho^\eps\,dx.
\]
Because $\chi\nabla V\in C_b^2$, \eqref{conc} controls the first term by $C_T\mathcal D^\eps(t)$. Using \eqref{conc} with $1-\chi$ gives $\int(1-\chi)\rho^\eps\lesssim_T\mathcal D^\eps(t)$. Also, \eqref{l35e} and $E_0(\psi^\eps)-E_0(P_\omega)\ge0$ give $\int(1-\chi)V\rho^\eps\lesssim_T\mathcal D^\eps(t)$. By Lemma~\ref{l36} and Cauchy--Schwarz,
\[
\int(1-\chi)|\nabla V|\rho^\eps
\le
\left(\int(1-\chi)|\nabla V|^2\rho^\eps\right)^{1/2}
\left(\int(1-\chi)\rho^\eps\right)^{1/2}
\lesssim_T\mathcal D^\eps(t).
\]
Thus $|\dot\sigma^\eps(t)|\lesssim_T\mathcal D^\eps(t)$. Next, by $\dot x=2\xi$,
\[
\dot\lambda^\eps(t)=2m_\omega\xi(t)\cdot\nabla V(x(t))-\int\nabla(\chi V)(x)\cdot p^\eps(t,x)\,dx.
\]
Since $\nabla(\chi V)(x(t))=\nabla V(x(t))$,
\[
\dot\lambda^\eps(t)=\int\nabla(\chi V)(x)\cdot[2m_\omega\xi(t)\delta_{x(t)}-p^\eps(t,x)\,dx].
\]
As $\nabla(\chi V)\in C_b^2$, \eqref{conc} yields $|\dot\lambda^\eps(t)|\lesssim_T\mathcal D^\eps(t)$. Finally, set $a(x):=x\chi(x)$. By the continuity equation, we have that
\[
\dot\gamma^\eps(t)=2m_\omega\xi(t)-\int Da(x)p^\eps(t,x)\,dx,
\]
where $Da$ is the Jacobian of $a$. Since $Da(x(t))=I_3$,
\[
\dot\gamma^\eps(t)=\int Da(x)[2m_\omega\xi(t)\delta_{x(t)}-p^\eps(t,x)\,dx].
\]
Applying \eqref{conc} row by row to $Da\in C_b^2$ gives $|\dot\gamma^\eps(t)|\lesssim_T\mathcal D^\eps(t)$. This proves the claim. Combining and integrating,
\[
\Xi^\eps(t)\lesssim_T\eps^2+\int_0^t\Xi^\eps(s)\,ds,\qquad 0\le t<\tau_\eps^*.
\]
Gronwall lemma gives
\begin{align}
	\sup_{0\le t<\tau_\eps^*}\Xi^\eps(t)\lesssim_T\eps^2. \label{xibd}
\end{align}

By continuity, this holds at $t=\tau_\eps^*$. Shrinking $\eps_T$, assume $\Xi^\eps(\tau_\eps^*)<\tfrac12h_0$. If $\tau_\eps^*<T$, continuity would extend the bootstrap condition beyond $\tau_\eps^*$, a contradiction. Hence $\tau_\eps^*=T$ and $\sup_{0\le t\le T}\Xi^\eps(t)\lesssim_T\eps^2$.

Thus,  \eqref{ECC} holds on $[0,T]$, so Proposition~\ref{p51} applies on $[0,T]$, giving continuous modulation parameters and $\sup\|w_1^\eps(t)\|_{H^1_\eps}\lesssim_T\eps$. Lemma~\ref{l54}, first on $[0,T)$ and then at $T$ by continuity, gives $\sup|y^\eps(t)-x(t)|\lesssim_T\eps^2$.

Set $d^\eps(t):=(y^\eps(t)-x(t))/\eps$. Then $\sup|d^\eps(t)|\lesssim_T\eps$. Since $P_\omega\in H^2$,
\[
\|P_\omega(\cdot-d^\eps(t))-P_\omega\|_{H^1}\lesssim_\omega |d^\eps(t)|\lesssim_T\eps.
\]
The uniform boundedness of $\xi$ then gives
\[
\left\|e^{\frac{i}{\eps}[\xi(t)\cdot x+\eps\theta_*^\eps(t)]}
\left[P_\omega\!\left(\frac{x-y^\eps(t)}{\eps}\right)-P_\omega\!\left(\frac{x-x(t)}{\eps}\right)\right]\right\|_{H^1_\eps}
\lesssim_T\eps.
\]

Define
\[
w^\eps(t,x):=w_1^\eps(t,x)
+e^{\frac{i}{\eps}[\xi(t)\cdot x+\eps\theta_*^\eps(t)]}
\left[P_\omega\!\left(\frac{x-y^\eps(t)}{\eps}\right)-P_\omega\!\left(\frac{x-x(t)}{\eps}\right)\right].
\]
Then $w^\eps\in C([0,T];H^1_\eps)$ and $\sup\|w^\eps(t)\|_{H^1_\eps}\lesssim_T\eps$.

Finally, set $\theta^\eps(t):=\theta_*^\eps(t)+\frac1\eps\xi(t)\cdot x(t)$. This is continuous and
\[
\frac{i}{\eps}[\xi(t)\cdot x+\eps\theta_*^\eps(t)]
=
\frac{i}{\eps}\xi(t)\cdot(x-x(t))+i\theta^\eps(t).
\]
Hence
\[
u^\eps(t,x)=e^{\frac{i}{\eps}\xi(t)\cdot(x-x(t))+i\theta^\eps(t)}P_\omega\!\left(\frac{x-x(t)}{\eps}\right)+w^\eps(t,x),
\]
which proves \eqref{deco} and \eqref{err}. This completes the proof of theorem.
\end{proof}


\section{Proofs of the corollaries}
\label{pcor}

We conclude by deriving the concentration of mass and momentum
and the approximation of the barycenter. Both statements are
direct consequences of the quantitative defect estimate
\eqref{xibd} obtained in the proof of Theorem~\ref{main}.

\begin{proof}[Proof of Corollary~\ref{mcon}]
Fix $T>0$. Lemma~\ref{l55} and \eqref{xibd} give
\[
\left\|\rho^\eps(t)\,dx-m_\omega\delta_{x(t)}\right\|_{(C_b^2)^*}
+
\left\|p^\eps(t)\,dx-2m_\omega\xi(t)\delta_{x(t)}\right\|_{(C_b^2)^*}
\lesssim_T\Xi^\eps(t)+\eps^2\lesssim_T\eps^2
\]
uniformly for $0\le t\le T$, which is \eqref{meas}.
\end{proof}

\begin{proof}[Proof of Corollary~\ref{bary}]
Since $u^\eps(t)\in\Sigma$, the barycenter is well defined. A standard cutoff approximation to the coordinate functions, together with the uniform $L^1$ bound for $p^\eps$ from Lemma~\ref{l31}, gives
\[
\frac{d}{dt}\int x\rho^\eps(t,x)\,dx=\int p^\eps(t,x)\,dx.
\]
Thus $\dot b^\eps(t)=\frac1{m_\omega}\int p^\eps(t,x)\,dx$. Since $\dot x(t)=2\xi(t)$ and $\sigma^\eps(t)=\int p^\eps-2m_\omega\xi(t)$, we get $\dot b^\eps(t)-\dot x(t)=\frac1{m_\omega}\sigma^\eps(t)$.

At $t=0$, $b^\eps(0)=\frac1{m_\omega}\int(x_0+\eps y)P_\omega^2(y)\,dy=x_0=x(0)$, since $P_\omega$ is radial. Hence, by \eqref{xibd},
\[
|b^\eps(t)-x(t)|\le\frac1{m_\omega}\int_0^t|\sigma^\eps(s)|\,ds\lesssim_T\eps^2.
\]
Taking the supremum over $0\le t\le T$ gives \eqref{bar2}.
\end{proof}

\appendix

\section{Proof of Proposition~\ref{wp}}
\label{sec:cauchy}

As the parameter $\eps$ plays no role in the present Cauchy theory, we
shall assume here that 
$\eps=1$, and consider
\begin{equation}
  \label{eq:NLSPcub-quint}
  i\d_t u+\Delta u = V(x) u -|u|^2u+|u|^4u,\quad x\in \R^3\quad;\quad
  u_{\mid t=0}=u_0\in \Sigma. 
\end{equation}
We explain why Proposition~\ref{wp} can be inferred from \cite{CW89},
\cite{Jao2016,Jao2018} and \cite{Zhang}. 
\smallbreak

We note that unless $V$ grows quadratically in all the directions, the
space $C(\R;\Sigma)$ involved in Proposition~\ref{wp} is smaller than
the space $C(\R;\H^1)$, where
\begin{equation*}
 \H^1:= \left\{ u\in H^1(\R^3), \int_{\R^3}V(x)|u(x)|^2dx<\infty \right\}
\end{equation*}
is the largest space ensuring that mass $M$ and energy
$E_{V}$ are
well-defined. Note that in order to properly
define the center of mass $b^\eps(t)$ involved in
Corollary~\ref{bary}, the above energy space may not suffice. In
addition, it is not clear that our strategy to construct the solution $u$ to
\eqref{eq:NLSPcub-quint} allows to work in a larger space than
$C(\R;\Sigma)$. Indeed, some results from \cite{Jao2018}, which we
use here, rely of the parametrix construction from
\cite{Fujiwara79,Fujiwara},  and involve quite naturally the space
$\Sigma$, when the behavior of $V$ at infinity is not more specified
than in Assumption~\eqref{H1}. This can be seen for instance by
inspecting the proof of \cite[Lemma~4.8]{Jao2018}, which measures the
interplay between $\dot H^1$-critical scaling transforms, and
$e^{-it(-\Delta+V)}$ viewed as a Fourier integral operator.  

Consider the Hamiltonian
\begin{equation*}
  H= -\Delta+V.
\end{equation*}
By using a gauge transform in
\eqref{eqn} (replace $u$ with $u e^{-it(1+\inf V)}$, a transform which
does not affect Proposition~\ref{wp}), we may
assume $V\ge 1$ in \eqref{eq:NLSPcub-quint}. 

One technical drawback in working in $\Sigma$  is that, unlike
$H^{1/2}$, the operator $H^{1/2}_{\rm harmo}$, corresponding to
the harmonic potential $|x|^2$ in place of $V$, does not commute with
$S_V(t)$. Instead of $ H^{1/2}_{\rm harmo}$, we consider the
operators $\nabla$ and (multiplication by) $x$. We note that if $w$
solves
\begin{equation*}
  i\d_t w=H w +F, 
\end{equation*}
then $\nabla w$ and $x w$ solve
\begin{equation*}
  \left\{
    \begin{aligned}
&      i\d_t \nabla w = H \nabla w +w\nabla V +\nabla F,\\
&      i\d_t x w = H x w -\nabla w +x F.
    \end{aligned}
  \right.
\end{equation*}
Since $V$ satisfies \eqref{H1}, we have the pointwise bound $|\nabla
V(x)|\lesssim 1+|x|$, so when considering energy estimates or
estimates based on Strichartz norms, we obtain a closed system of
estimates, and $\|\nabla w\|_Y+\|x w\|_Y$ is estimated in the same
fashion as $\|\nabla w\|_Y$ in the case $V=0$, thanks to an
additional Gronwall type argument.

\subsection{Technical preliminaries}

From now on, we suppose that $V$ satisfies
Assumption~\eqref{H1} and $V\ge 1$,  where the latter comes for free in
the proof of Proposition~\ref{wp}, as noted above. 

We first recall that \cite[Proposition~2]{Jao2018}, involving some
equivalence of 
norms when in addition to Assumption~\eqref{H1}, 
$V(x)\gtrsim |x|^2$, remains valid under Assumption~\eqref{H1} only, as
proven in \cite[Lemma~2.12]{Ca26} ($V\ge 1$). As noted in \cite{Ca26},
this result can be proven either by complex interpolation, or by
Weyl-H\"ormander pseudodifferential calculus, thanks in particular to
\cite{Beals1979,HelfferNier2005} (using $s\in [0,1]$ in
\cite[Theorem~4.8]{HelfferNier2005} and not only $s\in \{1/2,1\}$ as
in \cite[Remark~2.13]{Ca26}).

\begin{lemma}[Equivalence of norms]\label{lem:equiv-norms}
  Let $V$ satisfying Assumption~\eqref{H1} with in addition $V(x)\ge
  1$ for all $x\in \R^3$. For any $1< p <\infty$ and $0\le s\le 1$,
  there exist $C_1,C_2>0$ such that for any 
$\phi\in\Sch(\R^d)$,
\begin{equation*}
  \|(-\Delta)^s \phi\|_{L^p}+ \|V^s \phi \|_{L^p}\le C_1
  \|H^{s} \phi\|_{L^p}\le C_2 
  \(\|(-\Delta)^s \phi\|_{L^p}+ \|V^s \phi \|_{L^p}\).
\end{equation*}
\end{lemma}
We define the norm associated to the energy space as follows:
\begin{equation*}
  \|u\|_{\H^1}^2 = \<Hu,u\> =\int_{\R^3} \(|\nabla
  u(x)|^2+V(x)|u(x)|^2\)dx= \|\nabla u\|_{L^2}^2 + \|V^{1/2}u\|_{L^2}^2.
\end{equation*}
Since $V\ge 1$ is at most quadratic, there exists $C>0$ such that 
\begin{equation}\label{eq:compare-norms}
 \|u\|_{H^1}\le  \|u\|_{\H^1}\le C\|u\|_\Sigma, \quad \forall u\in \Sigma. 
\end{equation}
We also state formally the above remark concerning the propagation
of the first momentum, in the linear case first. 
\begin{lemma}\label{lem:H1-Sigma}
  Let $I\ni t_0$ be a bounded time interval. There exists $C>0$ such
  that 
  \begin{equation*}
    \sup_{t\in I}\|e^{-i(t-t_0)H}\phi\|_\Sigma \le C
    (1+|I|)\|\phi\|_\Sigma,\quad \forall \phi\in \Sigma. 
  \end{equation*}
\end{lemma}
\begin{proof}
  We obviously have
  \begin{align*}
    \frac{d}{dt}\|e^{-i(t-t_0)H}\phi\|_{\H^1}^2
    &=\frac{d}{dt}\<H
      e^{-i(t-t_0)H}\phi, e^{-i(t-t_0)H}\phi\> \\
    &= \frac{d}{dt}\<e^{-i(t-t_0)H}H
    \phi, e^{-i(t-t_0)H}\phi\> =0.
  \end{align*}
Let $w(t) = e^{-i(t-t_0)H}\phi$.  From the above remark
  \begin{equation*}
    i\d_t xw= H x w -\nabla
    w, 
  \end{equation*}
  so we infer by $L^2$-energy estimate,
  \begin{align*}
    \|x w(t)\|_{L^2}
    &\le
    \|x\phi\|_{L^2}+\left|\int_{t_0}^t \|\nabla
      w\|_L^2ds\right|\\
    &\le \|x\phi\|_{L^2}+
      \left|\int_{t_0}^t \<H
      e^{-i(s-t_0)H}\phi, e^{-i(s-t_0)H}\phi\>^{1/2}ds\right| \\
    &\le \|x\phi\|_{L^2}
    + |t-t_0| \|\phi\|_{\H^1},
  \end{align*}
  hence the lemma, in view of \eqref{eq:compare-norms}. 
\end{proof}
In view of Assumption~\eqref{H1}, \cite{Fujiwara79,Fujiwara} guarantees
that the $L^2$-unitary propagator $S_V(t) = e^{-itH}$ enjoys local
dispersive  estimates in the sense that there exists $\tau>0$ such
that
\begin{equation}\label{eq:dispV}
  \left\| e^{-itH}\right\|_{L^1(\R^3)\to L^\infty(\R^3)}\lesssim
  \frac{1}{|t|^{3/2}},\quad |t|\le \tau. 
\end{equation}
As evoked in Remark~\ref{rem:smoothness}, the proof presented in
\cite{Fujiwara79,Fujiwara} relies on the boundedness of the family
$(\d^\alpha V)_{2\le |\alpha|\le M}$, for some finite $M$ whose value
is not examined. In particular, Strichartz estimates, including the
endpoint case, are available, provided that only bounded time
intervals are considered (the case of the harmonic oscillator shows
that $ H$ may have eigenfunctions, whose evolution under $S_V(t)$
is a solitary wave). 
\smallbreak

We see
that the arguments from the proof of \cite[Theorem~1.1]{Zhang} can be
repeated: if we can solve 
\begin{equation}
\label{eq:NLSPquint}
  i\d_t \u+\Delta \u = V(x) \u +|\u|^4\u,\quad x\in \R^3\quad;\quad
  \u_{\mid t=0}=u_0\in \Sigma,
\end{equation}
then the cubic term in \eqref{eq:NLSPcub-quint} can be viewed as a
subcritical perturbation. In the above mentioned argument, the
property $\u\in C(\R;\Sigma)$ does not suffice: the property
\begin{equation*}
\u,\nabla \u,x\u\in L^q_{\rm loc}(\R;L^r(\R^3))\quad\text{for all
    admissible pair }(q,r),
\end{equation*}
is needed too. To show this holds true, we rely on the local theory
from \cite{CW89} (as presented in \cite{CazCourant}), and
\cite{KV-clay} (in particular for the condition $(iv)$ below). A
direct adaptation of 
\cite[Theorem~4.5.1]{CazCourant} to include the potential $V$ in the
Cauchy problem in $\Sigma$ yields:
\begin{lemma}\label{lem:local-Sigma}
  Let $V$ satisfying \eqref{H1}, and $u_0\in \Sigma$. There exists a
  unique strong $\Sigma$-solution $\u$ of \eqref{eq:NLSPquint} defined
  on the maximal interval $(-T_{\rm min},T_{\rm max})$, with $0<T_{\rm
    min},T_{\rm max}\le \infty$. Moreover, the following properties
  hold.
  \begin{itemize}
  \item[(i)] There is conservation of mass and energy.
  \item[(ii)] $\u,\nabla \u,x\u\in L^q_{\rm loc}((-T_{\rm min},T_{\rm max});L^r(\R^3))$ for
    every admissible pair, that is $(q,r)$ such that
    \begin{equation*}
      2\le r\le 6,\quad \frac{2}{q}=3\(\frac{1}{2}-\frac{1}{r}\). 
    \end{equation*}
  \item[(iii)] If $T_{\rm max}<\infty$, then
    \begin{equation*}
      \|\nabla \u\|_{L^q((0,T_{\rm max});L^r(\R^3))} =\infty 
    \end{equation*}
    for every admissible pair $(q,r)$ such that $2<r<3$.
  \item[(iv)] If $T_{\rm max}<\infty$, then
    \begin{equation*}
      \|\u\|_{L^6((0,T_{\rm max});L^{18}(\R^3))} =\infty .
    \end{equation*}
  \end{itemize}
\end{lemma}
We note that in the above statement $(iii)$, only the gradient is involved in
the obstruction for global existence. Indeed, we have easily the
blow-up condition
\begin{equation*}
 \|\nabla \u\|_{L^q((0,T_{\rm max});L^r)}  + \|x
      \u\|_{L^q((0,T_{\rm max});L^r)} =\infty .
\end{equation*}
However, as $x\u$ solves
\begin{equation*}
  i\d_t x\u = Hx\u -\nabla \u + |\u|^4x\u,
\end{equation*}
we may use Strichartz estimates with the same numerology as in the
case $V=0$ 
(see e.g. \cite{KV-clay}), to check that if $\|\nabla
\u\|_{L^q((0,T_{\rm max});L^r)} $ is finite for all admissible pair,
then so is $\|x\u\|_{L^q((0,T_{\rm max});L^r)} $. Alternatively,
the norm in condition $(iv)$ is controlled, in view of Sobolev
embedding, as
\begin{equation*}
  \|u(t)\|_{L^{18}(\R^3)} \lesssim \|\nabla u(t)\|_{L^{18/7}(\R^3)},
\end{equation*}
and we note that the pair $(6,\tfrac{18}{7})$ is admissible. We have
similarly: 
\begin{lemma}\label{lem:local-H1}
  Let $V$ satisfying \eqref{H1}, and $u_0\in \H^1$. There exists a
  unique strong $\H^1$-solution $\u$ of \eqref{eq:NLSPquint} defined
  on the maximal interval $(-T_{\rm min},T_{\rm max})$, with $0<T_{\rm
    min},T_{\rm max}\le \infty$. Moreover, the following properties
  hold.
  \begin{itemize}
  \item[(i)] There is conservation of mass and energy.
  \item[(ii)] $\u,\nabla \u,V^{1/2}\u\in L^q_{\rm loc}((-T_{\rm
      min},T_{\rm max});L^r(\R^3))$ for 
    every admissible pair.
  \item[(iii)] If $T_{\rm max}<\infty$, then
    \begin{equation*}
      \|\nabla \u\|_{L^q((0,T_{\rm max});L^r(\R^3))} =\infty 
    \end{equation*}
    for every admissible pair $(q,r)$ such that $2<r<3$.
  \item[(iv)] If $T_{\rm max}<\infty$, then
    \begin{equation*}
      \|\u\|_{L^6((0,T_{\rm max});L^{18}(\R^3))} =\infty .
    \end{equation*}
  \end{itemize}
\end{lemma}
Formally, the proof is even closer to that of the case without
potential than for Lemma~\ref{lem:local-Sigma}, where a coupling
between $\nabla \u$ and $x\u$ is present. Indeed, when considering the
Duhamel's formula
\begin{equation*}
  \u(t) = S_V(t)u_0-i\int_0^t S_V(t-\sigma)\(|\u|^4\u\)(\sigma)d\sigma,
\end{equation*}
one applies the operator $H^{1/2}$ (defined by functional calculus in
general), which commutes with $S_V$,
\begin{equation*}
  H^{1/2}\u(t) = S_V(t)H^{1/2} u_0-i\int_0^t
  S_V(t-\sigma)H^{1/2}\(|\u|^4\u\)(\sigma)d\sigma. 
\end{equation*}
The action of $H^{1/2}$ on the nonlinear term is not explicit in
general, but thanks to 
Lemma~\ref{lem:equiv-norms}, all $L^q_tL^r_x$ estimates are the same as
in the case $V=0$, up to replacing the gradient by $H^{1/2}$, and
changing the multiplicative constants. 
A more sophisticated version of  Lemma~\ref{lem:H1-Sigma} reads as
follows:
\begin{lemma}\label{lem:equiv-space}
  Let $V$ satisfying \eqref{H1}, and $u_0\in \Sigma$. If $\u$ solves
  \eqref{eq:NLSPquint} in the sense of Lemma~\ref{lem:local-H1}, then
  it is also a solution in the sense of Lemma~\ref{lem:local-Sigma},
  that is, $\H^1$-regularity is upgraded to $\Sigma$-regularity.
\end{lemma}
We note that the standard stability argument involved in critical
problems is also 
readily adapted to the present case: \cite[Proposition~4]{Jao2018} can
be resumed (with estimates in $\H^1$ or $\Sigma$, even though $\Sigma$
only is used eventually).  
\smallbreak

We consider a Littlewood-Paley decomposition associated to $V$ like
in \cite{Jao2018}. For $\psi\in C_0^\infty((0,\infty))$ so that
$1\equiv \sum_{N\text{ dyadic}}\psi(\lambda/N)$, defined the spectral
multipliers
\begin{align*}
  & P_N=\psi\(\frac{\sqrt{H}}{N}\),\qquad P_{\le N} = \sum_{M\le N}
    P_N,\\
  & \tilde P_N = e^{-H/2N^2} - e^{-2H/N^2},\quad \tilde P_{\le N} =
    \sum_{M\le N}  \tilde P_N.
\end{align*}
The operators $P_N$ are also considered in \cite{Ca26} thanks to
Weyl-H\"ormander calculus, from
\cite{Beals1979,HelfferRobert1983,HelfferNier2005}. Invoking in
addition \cite{Hebisch1990} like in \cite{Jao2018} (Theorem~2.3 there,
and the proof of Proposition~5), to
address $\tilde P_N$, we 
infer that the Bernstein estimates proven in \cite[Lemma~2.8]{Jao2016},
in the case $V(x)\equiv |x|^2$,
remain true for $V$ 
satisfying Assumption~\eqref{H1} and $V\ge 1$.

\subsection{Main steps of the argument}

As pointed out above, Proposition~\ref{wp} follows from its
counterpart where the cubic term is absent. The proof then follows the
same lines as in \cite{Jao2018}, which is an adaptation of
\cite{Jao2016}: in \cite{Jao2016}, $V(x)=|x|^2$, while in
\cite{Jao2018}, $V$ satisfies Assumption~\eqref{H1}, with in addition
the uniform estimate $V(x)\gtrsim |x|^2$. By invoking
Lemmas~\ref{lem:equiv-norms}, \ref{lem:H1-Sigma}
and \ref{lem:equiv-space}, we can remove the assumption $V(x)\gtrsim
|x|^2$, provided that we work with a $\Sigma$-regularity.
\smallbreak

Let $I\ni 0$ be a compact interval of length at most $\tau$ (see
\eqref{eq:dispV}).  
The notion of profile decomposition follows the same definitions as in
\cite{Jao2016,Jao2018}, up to a shade discussed below.

\begin{definition}
  A \emph{frame} is a sequence $(t_n,x_n,N_n)\in I\times\R^3\times
  2^\N$ conforming to one of the scenarios:
  \begin{enumerate}
  \item $N_n\equiv 1$, $t_n\equiv 0$, and $x_n\equiv 0$.
  \item $N_n\to \infty$ and $N_n^{-1}|x_n|\to r_\infty\in [0,\infty)$.
  \end{enumerate}
\end{definition}
In \cite{Jao2018}, the final property above is replaced by
$N_n^{-1}V(x_n)^{1/2}\to r_\infty\in [0,\infty)$, which turns out to
be equivalent since $V$ behaves quadratically. In the case we consider
here, we resume the property as stated in \cite{Jao2016}, associated to the
space $\Sigma$ (recall that for $V=0$, when working in $H^1$ or $\dot
H^1$, there is no reason to impose a control on $x_n$ -- this remark
is another reason why we work in $\Sigma$ rather than in $\H^1$: in
the case of a partial confinement, the energy space corresponds to
$H^1$ in the non-confined direction(s)).
\begin{definition}
  An \emph{augmented frame} is a sequence $(t_n,x_n,N_n,N'_n)\in
  I\times\R^3\times 2^\N\times \R$ belonging to one of the following
  types:
  \begin{enumerate}
  \item $N_n\equiv 1\equiv N'_n$, $t_n\equiv 0$, $x_n\equiv 0$.
  \item  $N_n\to \infty$ and $N_n^{-1}|x_n|\to r_\infty\in
    [0,\infty)$, and either
    \begin{itemize}
    \item[(a)] $N'_n\equiv 1$ if $r_\infty>0$, or
    \item[(b)] $N_n^{1/2}\le N'_n\le N_n$,
      $N_n^{-1}|x_n|\(\frac{N_n}{N'_n}\)\to 0$, and
      $\frac{N_n}{N'_n}\to \infty$, if $r_\infty=0$. 
    \end{itemize}
  \end{enumerate}
\end{definition}
The operators involved in the profile decomposition are defined by
\begin{equation*}
  \begin{aligned}
    (G_n\phi)(x)&= N_n^{1/2}\phi\(N_n(x-x_n)\),\\
    (\tilde G_nf)(t,x)&= N_n^{1/2}f\(N_n^2(t-t_n),N_n(x-x_n)\),
  \end{aligned}
\end{equation*}
and
\begin{equation*}
  S_n\phi=
  \begin{cases}
    \phi, & \text{ for frames of type 1 (i.e. }N_n\equiv 1\text{)},\\
    \chi\(\frac{N_n}{N'_n}\cdot \)\phi,
          & \text{ for frames of type 2  (i.e. }N_n\to\infty\text{)}, 
  \end{cases}
\end{equation*}
where $\chi\in C_0^\infty(\R^3)$ is equal to $1$ on the unit ball. We
recall the statement of \cite[Proposition~5]{Jao2018}, and explain why
it remains valid in our framework:
\begin{proposition}\label{prop:inv-stri}
Let $I\ni 0$ be a compact interval of length at most $\tau$, and
suppose $f_n$ is a sequence of functions in $\Sigma$ satisfying
\begin{equation*}
  0<\eps\le\|e^{-itH}f_n\|_{L^{10}_{t,x}(I\times\R^3)}
  \lesssim \|f_n\|_\Sigma\le A<\infty.
\end{equation*}
Then, after passing to a subsequence, there exists an augmented frame
\begin{equation*}
  \mathcal F =\{ (t_n,x_n,N_n,N'_n)\}
\end{equation*}
and a sequence of function $\phi_n\in \Sigma$ such that one of the
following holds:
\begin{enumerate}
\item $\mathcal F$ is of type 1 (i.e. $N_n\equiv 1$) and
  $\phi_n=\phi\in \Sigma$ is a weak limit of $f_n$ in $\Sigma$.
\item $\mathcal F$ is of type 2, either $t_n\equiv 0 $ or $N_n^2t_n\to
  \pm \infty$, and $\phi_n = e^{it_nH}G_nS_n\phi$ where $\phi\in \dot
  H^1(\R^3)$ is a weak limit of $G_n^{-1}e^{-it_nH}f_n$ in $\dot
  H^1$. Moreover, if $\mathcal F$ is of type 2a, then $\phi$ also
  belongs to $L^2(\R^3)$. 
\end{enumerate}
The functions $\phi_n$ ave the following properties:
\begin{align*}
  & \liminf_{n\to \infty} \|H^{1/2} \phi_n\|_{L^2}\gtrsim
    A\(\frac{\eps}{A}\)^{15/8}, \\
  &  \lim_{n\to \infty} \(\|f_n\|_{L^6}^6 - \|f_n-\phi_n\|_{L^6}^6 -
      \|\phi_n\|_{L^6}^6 \)=0 ,\\
  & \lim_{n\to \infty} \(\|f_n\|_\Sigma^2 - \|f_n-\phi_n\|_{\Sigma}^2 -
      \|\phi_n\|_{\Sigma}^2 \)=0.
\end{align*}
\end{proposition}
\begin{proof}[Sketch of the proof]
  Using Littlewood-Paley theory associated to $H$ (so $e^{-itH}$ and
  $P_N$ commute),
  \begin{equation*}
    \|e^{-itH} f \|_{L^{10}_{t,x}}\lesssim \|H^{1/2} f\|_{L^2}^{1/5}
    \sup_N \left\| e^{-itH} P_N f\right\|_{L^{10}_{t,x}}^{4/5}.
  \end{equation*}
  (The powers on the right hand side are swapped in the statement
  \cite[Proposition~4.1]{Jao2016}, but the proof indeed yields the
  above powers.) Recalling \eqref{eq:compare-norms} to obtain
  $\|H^{1/2}f_n\|_{L^2}\lesssim A$,  we infer that for each $n$,
  there exists $N_n$ such that 
  \begin{equation*}
    \left\|\tilde P_{N_n} e^{-itH} f_n\right\|_{L^{10}_{t,x}}\gtrsim
    \left\| P_{N_n} e^{-itH} f_n\right\|_{L^{10}_{t,x}}\gtrsim
    \eps^{5/4} A^{-1/4}. 
  \end{equation*}
  Interpolating $L^{10}$ between $L^{10/3}$ and $L^\infty$, invoking
  Strichartz and Bernstein estimates, there are sequences $t_n\in I$
  and $x_n\in \R^3$ such that
  \begin{equation*}
    \left| e^{-it_n H }\tilde P_{N_n} f_n(x_n)\right| \gtrsim
  N_n^{1/2}  A\(\frac{\eps}{A}\)^{15/8}.
\end{equation*}
The spatial weight in the space $\Sigma$ and pointwise heat kernel
bounds (as $V\ge 0$, by the maximum principle, one uses estimates from
the case $V=0$ like in \cite{Jao2016}) yield $|x_n|\le C(A,\eps)N_n$.
Using \eqref{eq:compare-norms}, one can then repeat the argument from
\cite{Jao2018}. 
\end{proof}
The profile decomposition stated in \cite[Proposition~4.14]{Jao2016} can
then be resumed. The modifications needed when the potential is not
exactly quadratic are  the object of Lemmas~4.7 and 4.8 there, which are
stated and proven under the assumption that $V$ is at most quadratic in
the sense of Assumption~\eqref{H1} (boundedness from below is not
needed at this stage). The rest of the argument from
\cite[Section~6]{Jao2016} can then be repeated.

\end{document}